\documentclass[a4paper]{amsart} 

\newcommand{\f}{\frac}

\newcommand{\del}{\partial}

\newcommand{\im}{\operatorname{Im}}

\newcommand{\R}{\mathbb R}
\newcommand{\C}{\mathbb C}
\newcommand{\N}{\mathbb N}
\newcommand{\Z}{\mathbb Z}

\newcommand{\re}{\operatorname{Re}}
\newcommand{\eps}{\varepsilon}
\renewcommand{\epsilon}{\varepsilon}

\newcommand{\Om}{\Omega}

\newcommand{\dist}{\operatorname{dist}}

\newcommand{\h}{\mathcal{H}}

\newenvironment{enuma}{\begin{enumerate}[(a)]}{\end{enumerate}}
\newenvironment{enumi}{\begin{enumerate}[(i)]}{\end{enumerate}}

\usepackage{etex}
\usepackage[english]{babel}
\usepackage{tikz}
\usetikzlibrary{arrows,decorations.pathmorphing,backgrounds,positioning,fit,petri,patterns}
\usepackage{pgfplots}
\usepackage[utf8]{inputenc}
\usepackage{graphicx}
\usepackage[colorlinks=true,linkcolor=red,citecolor=blue]{hyperref}
\usepackage{amsfonts}
\usepackage{enumerate}
\usepackage{booktabs}
\usepackage{array} 
\usepackage{paralist} 
\usepackage{subfig} 
\usepackage{mathtools}
\usepackage{tabu}
\usepackage{amsthm}
\usepackage{empheq}
\usepackage{amsopn}
\usepackage{dsfont}
\usepackage{wrapfig}
\usepackage[font=small]{caption}
\usepackage{units}
\usepackage[symbol]{footmisc}
\usepackage{todonotes}
\usepackage{amssymb}
\usepackage{pdfpages}
\usepackage{setspace}
\usepackage{float}

\allowdisplaybreaks
\numberwithin{equation}{section}

\theoremstyle{definition}
\newtheorem{de}{Definition}[section]

\theoremstyle{plain}
\newtheorem{prop}[de]{Proposition}
\newtheorem{lemma}[de]{Lemma}
\newtheorem{theorem}[de]{Theorem}

\theoremstyle{remark}
\newtheorem{remark}[de]{Remark}

\usepackage[ruled,vlined]{algorithm2e}
\usepackage{soul}
\usepackage{multirow,booktabs}
\usepackage{scrextend}
\usepackage[outline]{contour}

\newcommand{\SCI}{\operatorname{SCI}}

\renewcommand{\i}{\mathrm{i}}
\newcommand{\Lin}{M_{\mathrm{inner}}}
\newcommand{\Lout}{M_{\mathrm{outer}}}
\newcommand{\diag}{\operatorname{diag}}
\newcommand{\HD}{H_{\mathrm D}}
\newcommand{\Linull}{M_{\mathrm{inner},0}}
\newcommand{\BJ}{\mathcal{J}}
\newcommand{\BH}{\mathcal{H}}
\newcommand{\Res}{\operatorname{Res}}
\newcommand{\cO}{\mathcal{O}}

\title{Computing the sound of the sea in a seashell}
\author{Jonathan Ben-Artzi}
\email{Ben-ArtziJ@cardiff.ac.uk}
\author{Marco Marletta}
\email{MarlettaM@cardiff.ac.uk}
\author{Frank R\"{o}sler}
\email{RoslerF@cardiff.ac.uk}
\thanks{JBA acknowledges support from an Engineering and Physical Sciences Research Council Fellowship EP/N020154/1 and MM acknowledges support from an Engineering and Physical Sciences Research Council Grant EP/T000902/1. FR acknowledges support from the European Union's Horizon 2020 Research and Innovation Programme under the Marie Sklodowska-Curie grant agreement No. 885904. The authors thank the anonymous referees for their insightful comments which helped improve the presentation of this paper.}
\thanks{Communicated by Arieh Iserles}
\address{School of Mathematics, Cardiff University, Senghennydd Road, Cardiff CF24 4AG, Wales, UK}
\date\today

\begin{document}

\maketitle

\begin{abstract}
	The question of whether there exists an approximation procedure to compute the resonances of any Helmholtz resonator, regardless of its particular shape, is addressed. A positive answer is given, and it is shown that all that one has to assume is that the resonator chamber is bounded and that its boundary is $\mathcal C^2$. The proof is constructive, providing a universal algorithm which only needs to access the values of the characteristic function of the chamber at any requested point.
\end{abstract}
%
%
\section{Introduction}
 This paper provides an affirmative answer to the following question:
\begin{quote}\emph{Does there exist a  universal algorithm for computing the  resonances of the Laplacian in $\R^2\setminus\overline{U}$ for any open bounded set $U\subset\R^2$?}
\end{quote}
Any domain $U\subset\R^d$ gives rise to resonances, i.e. special frequencies that are `nearly' eigenvalues of the Laplacian in $\R^d\setminus\overline{U}$ (with appropriate boundary conditions). The most famous example is the ``sound of the sea'' in a seashell:  when we hold a seashell against our ear we hear  frequencies that are nearly the eigenvalues of the Laplacian in the closed cavity with Neumann boundary conditions. This is also known as a \emph{Helmholtz resonator} \cite{H1863}. We therefore refer to $U$ as a ``resonator''.

We are interested in computing (Dirichlet) resonances in a way that is independent of the domain $U$ itself; i.e. $U$ is the input of the problem, and the computation returns the associated resonances. We work in the plane, i.e. $d=2$, but our results can be adapted to higher dimensions.	
To our best knowledge this is the first time this question is addressed. Furthermore, the proof of existence  provides an actual algorithm (that is, the proof is constructive). We test this algorithm on some standard examples, and compare to known results.

The framework required for this analysis is furnished by the \emph{Solvability Complexity Index} ($\SCI$), which is an abstract theory for the classification of the computational complexity of problems that are infinite-dimensional. This framework has been developed over the last decade by Hansen and collaborators (cf. \cite{Hansen11,AHS}) and draws inspiration from   the seminal result \cite{DM} on solving quintic equations via a \emph{tower of algorithms}. {\bf We therefore emphasize that ours is an abstract result in analysis,  not  in numerical analysis.}
\subsection{Resonances}\label{sec:background}
Let us first define what we mean by a resonance. 
Let $U\subset\R^d$ be an open set and assume that  $\partial U\in\mathcal{C}^2$. Let $H$ be the Laplacian in $L^2(\R^d\setminus\overline U)$ with homogeneous Dirichlet boundary conditions on $\del U$. Resonances of $H$ can be defined via analytic continuation of the associated Dirichlet-to-Neumann (DtN) operators. Indeed, we can start from a characterization of eigenvalues and perform analytic continuation as follows. Denote the spectral parameter by $k^2$ (with the branch cut of the complex square root running along the positive real line). Then $k^2$ is an eigenvalue of $H$ if and only if there exists a function $u\in L^2(\R^d\setminus\overline U)$, such that $(-\Delta-k^2)u=0$ in $\R^d\setminus\overline U$ and $u=0$ on $\del U$. The existence of such a $u$ is equivalent to the following. 
Let $R>0$ be such that $U\subset B_R$ (the open ball of radius $R$ around 0). Denote by $M_\text{inner}(k)$ and $M_\text{outer}(k)$, respectively, the inner and outer DtN maps on $\del B_R$ associated with $-\Delta-k^2$. $M_\text{inner}(k)$ and $M_\text{outer}(k)$ can be shown to be analytic operator-valued functions of $k$ in the upper half plane $\C^+:=\{z\in\C\,|\,\im(z)>0\}$. If there exists an eigenfunction $u$ of $H$ with eigenvalue $k^2$, then
\begin{align*}
	u|_{\del B_R}\in \ker(M_\text{inner}(k)+M_\text{outer}(k)).
\end{align*}
Conversely, if $\phi\in \ker(M_\text{inner}(k)+M_\text{outer}(k))$, then there exist solutions $u_{\text{inner}}$ and $u_\text{outer}$ of $(-\Delta-k^2)u=0$ in $B_R\setminus \overline U$ and $\R^d\setminus \overline{B}_R$ respectively, such that $u_{\text{inner}}|_{\del B_R}=u_{\text{outer}}|_{\del B_R}=\phi$ and $\del_\nu u_{\text{inner}}|_{\del B_R}=-\del_\nu u_{\text{outer}}|_{\del B_R}$ (note that the normal vector $\nu$ is always taken to point to the exterior of the domain considered, that is, away from 0 for $B_R$ and towards 0 for $\R^d\setminus B_R$). Hence, the function
\begin{align}\label{eq:gluing_u}
	u:=\begin{cases}
		u_\text{inner} & \text{ on } B_R\setminus\overline U\\
		u_\text{outer} & \text{ on } \R^d\setminus \overline{B}_R
	\end{cases}
\end{align}
is in $H^2(\R^d\setminus \overline{U})$.
This shows that a complex number $k\in\C^+$ in the upper half plane is a (square root of an) eigenvalue if and only if $\ker(M_\text{inner}(k)+M_\text{outer}(k))\neq\emptyset$.
A \emph{resonance} can now be defined as follows. 
\begin{de}[Resonance]\label{def:res}
	Let us use the same symbols $M_\text{inner}(\cdot),\,M_\text{outer}(\cdot)$ to denote the meromorphic continuations of the DtN maps to all of $\C$. A number $k\in\C^-:=\{z\in\C\,|\,\im(z)<0\}$ in the lower half plane is called a \emph{resonance} of $H$, if 
	\begin{align}\label{eq:Ker=0}
		\ker(M_\text{inner}(k)+M_\text{outer}(k))\neq\{0\}.
	\end{align}
\end{de}
\begin{remark}
	Resonances in the sense of Definition \ref{def:res} are well-defined. Indeed, $M_{\text{outer}}(k)$ is well-defined for all $k\in\C$ and the only poles of $M_\text{inner}(k)$ are equal to the Dirichlet eigenvalues of $B_R\setminus \overline U$. But these all lie on the real axis, which is excluded in Definition \ref{def:res}. 

	Moreover, an argument similar to the one surrounding eq. \eqref{eq:gluing_u} shows that the nontriviality of $\ker(M_\text{inner}(k)+M_\text{outer}(k))$ is independent of the value of $R$ as long as $U\subset B_{R}$.
\end{remark}
\begin{remark}
	In two dimensions, the DtN maps $M_\text{inner},M_\text{outer}$ can actually be analytically continued to the \emph{logarithmic cover} of $\C$, rather than just $\C^-$ (cf. \cite[Sec. 3.1.4]{DZ}). However, since the vast majority of the literature is concerned with resonances near the real axis, we decided to simplify our presentation by considering only resonances in $\C^-$. A strategy for dealing with the general case has been outlined in \cite[Sec. 4.2.2]{BMR2020}.
\end{remark}
We show that resonances can be computed as the limit of a sequence of approximations, each of which can be computed precisely using finitely many arithmetic operations and accessing finitely many values of the characteristic function $\chi_U$ of $U$. The proof is constructive: we define an algorithm and prove its convergence. We emphasize that this \emph{single} algorithm is valid for \emph{any} open $U\subset\R^2$ with $\partial U\in\mathcal C^2$.  We  implement this algorithm and compare its output to known results. 
\subsection{The Solvability Complexity Index}\label{subsec:sci}
The Solvability Complexity Index (SCI) addresses questions which are at the nexus of pure and applied mathematics, as well as computer science: 
 	\begin{quote}\emph{How do we compute objects that are ``infinite'' in nature if we can only handle a finite amount of information and perform finitely many mathematical operations? Indeed, what do we even mean by ``computing'' such an object?}
	\end{quote}
These broad topics are addressed in the sequence of papers \cite{Hansen11,AHS,Ben-Artzi2015a}. Let us summarize the main definitions and discuss how these relate to our problem of finding resonances:
\begin{de}[Computational problem]\label{def:computational_problem}
	A \emph{computational problem} is a quadruple $(\Om,\Lambda,\Xi,\mathcal M)$, where 
	\begin{enumi}
		\item $\Om$ is a set, called the \emph{primary set},
		\item $\Lambda$ is a set of complex-valued functions on $\Om$, called the \emph{evaluation set},
		\item $\mathcal M$ is a metric space,
		\item $\Xi:\Om\to \mathcal M$ is a map, called the \emph{problem function}.
	\end{enumi}
\end{de}
\begin{de}[Arithmetic algorithm]\label{def:Algorithm}
	Let $(\Om,\Lambda,\Xi,\mathcal M)$ be a computational problem. An \emph{arithmetic algorithm} is a map $\Gamma:\Om\to\mathcal M$ such that for each $T\in\Om$ there exists a finite subset $\Lambda_\Gamma(T)\subset\Lambda$ such that
	\begin{enumi}
		\item the action of $\Gamma$ on $T$ depends only on $\{f(T)\}_{f\in\Lambda_\Gamma(T)}$,
		\item for every $S\in\Om$ with $f(T)=f(S)$ for all $f\in\Lambda_\Gamma(T)$ one has $\Lambda_\Gamma(S)=\Lambda_\Gamma(T)$,
		\item the action of $\Gamma$ on $T$ consists of performing only finitely many arithmetic operations on $\{f(T)\}_{f\in\Lambda_\Gamma(T)}$.
	\end{enumi}
\end{de}
\begin{de}[Tower of arithmetic algorithms]\label{def:Tower}
	Let $(\Om,\Lambda,\Xi,\mathcal M)$ be a computational problem. A \emph{tower of algorithms} of height $k$ for $\Xi$ is a family $\Gamma_{n_1,n_2,\dots,n_k}:\Om\to\mathcal M$ of arithmetic algorithms such that for all $T\in\Om$
	\begin{align*}
		\Xi(T) = \lim_{n_k\to\infty}\cdots\lim_{n_1\to\infty}\Gamma_{n_1,n_2,\dots,n_k}(T).
	\end{align*}
\end{de}
\begin{de}[SCI]
	A computational problem $(\Om,\Lambda,\Xi,\mathcal M)$ is said to have a \emph{Solvability Complexity Index} ($\SCI$) of $k\in\N$ if $k$ is the smallest integer for which there exists a tower of algorithms of height $k$ for $\Xi$.
	If a computational problem has solvability complexity index $k$, we write \begin{align*}
 			\SCI(\Om,\Lambda,\Xi,\mathcal M)=k.
		 \end{align*}
If there exists a family $\{\Gamma_n\}_{n\in\N}$ of arithmetic algorithms and $N_1\in\N$ such that $\Xi=\Gamma_{N_1}$ then we define $\SCI(\Om,\Lambda,\Xi,\mathcal M)=0$.
\end{de}

\begin{remark}
One can even delve deeper into the $\SCI$ classification by considering the so-called \emph{$\SCI$ Hierarchy} which was introduced in \cite{AHS}. In a nutshell, this hierarchy considers not only how many limits a particular computational problem requires, but also whether one can establish   \emph{error bounds}. For the interested reader, we discuss this hierarchy (and why we are unable to obtain error bounds) in Appendix \ref{app:sci-hierarchy}.
\end{remark}

\subsection{Setting of the problem and main result}\label{sec:Setting}
Let us describe the elements comprising our computational problem, followed by our main theorem.\\

\emph{(i) \underline{The primary set $\Omega$}.}
We define
\begin{equation*}
	\Omega
	=
	\left\{
	\emptyset\neq U\subset\R^2\text{ open}\,|\,U\text{ is bounded and }\partial U\in\mathcal C^2
	\right\}.
\end{equation*}
~\\

\emph{(ii) \underline{The evaluation set $\Lambda$}.}
The evaluation set, which describes the data  at our algorithm's disposal, is comprised of (all points in) the set $U$:
\begin{align}\label{eq:evaluation_set}
	\Lambda &:= \{U\mapsto \chi_U(x)\,|\,x\in\R^2\}.
\end{align}
Providing the values of the characteristic functions $U\mapsto \chi_U(x)$ in $\Lambda$ means that for every $x\in\R^2$ we can test whether $x$ is included in $U$ or not.
\begin{remark}
	Our computations will involve not only the values $\chi_U(x)$, but also the values of the Bessel and Hankel functions $J_n(z)$, $H^{(1)}_n(z)$, $z\in\C$, $n\in\N$ as well as the exponentials $e^{\i n\theta}$, $\theta\in[0,2\pi)$, $n\in\N$. These do not have to be included as part of the evaluation set because  they can be approximated to arbitrary precision with explicit error bounds. In order to keep the presentation clear and concise, we will assume the values $J_n(z)$, $H^{(1)}_n(z)$, $e^{\i n\theta}$ are known and not track these explicit errors in our estimates.\\
\end{remark}

\emph{(iii) \underline{The metric space $\mathcal{M}$}.}
$\mathcal M$ is the space $\mathrm{cl}(\C)$ of all closed  subsets of $\C$ equipped with the Attouch-Wets metric, generated by the following distance function:

\begin{de}[Attouch-Wets distance] Let $A,B$ be non-empty closed sets in $\C$. The  \emph{Attouch-Wets distance} between them is defined as
\begin{align*}
	d_{\mathrm{AW}}(A,B) = \sum_{n=1}^\infty 2^{-n}\min\left\{ 1\,,\,\sup_{|x|<n}\left| \dist(x,A) - \dist(x,B) \right| \right\}.
\end{align*} 
\end{de}
Note that if $A,B\subset\C$ are bounded, then $d_{\mathrm{AW}}$ is equivalent to the Hausdorff distance $d_\mathrm{H}$. Furthermore, it can be shown (cf. \cite[Ch. 3]{beer}) that 
	\begin{equation}\label{eq:Attouch-Wets}
		d_\mathrm{H}(A_n\cap B, A\cap B)\to 0 \text{ for all }B\subset\C \text{ compact } \;\Rightarrow\; d_\mathrm{AW}(A_n,A)\to 0.
	\end{equation}
~\\

\emph{(iv) \underline{The problem function $\Xi$}.}
The problem function $\Xi(U)=\Res(U)$ is the map that associates to each $U\in\Omega$ the set of resonances (as defined in Definition \ref{def:res}) of $H$.\\

With these at hand, we can state our main theorem:
\begin{theorem}\label{th:mainth}
	There exists a family of arithmetic algorithms $\{\Gamma_n\}_{n\in\N}$ such that $\Gamma_n(U)\to\Res(U)$ as $n\to+\infty$ for any $U\in\Omega$, where the convergence is in the sense of the Attouch-Wets metric. That is, $\SCI(\Omega,\Lambda,\Res(\cdot),(\mathrm{cl}(\C),d_\mathrm{AW}))=1$.
\end{theorem}
\begin{remark}[The set $\Res(U)$ is not empty]
We note that whenever $U$ is not empty, the set $\Res(U)$ too is not empty, although the resonances can lie far below the real axis, depending on the geometry of $U$. We refer to \cite{Zworski1999} for an exposition on this matter.
\end{remark}
\begin{remark}[Neumann boundary conditions]
	In Subsection \ref{sec:neumann} we outline a strategy for adapting this result to Neumann boundary conditions, and implement a numerical model consisting of four circular holes: this is relevant for understanding resonances caused by oil-drilling platforms \cite{EP96}, for instance.
\end{remark}
Theorem \ref{th:mainth} is proved by identifying an operator of the form $\mathrm{Id}_{L^2}+\mathcal K(k)$, where $\mathcal K(k)$ is in some Schatten class $C_p$, with the property that $\ker(\mathrm{Id}_{L^2}+\mathcal K(k))\neq\{0\}$ if and only if
$\ker(M_\text{inner}(k)+M_\text{outer}(k))\neq\{0\}$. The former is equivalent to $\det_{\lceil p\rceil}(\mathrm{Id}_{L^2}+\mathcal K(k))=0$, where $\det_{\lceil p\rceil}$  denotes the $\lceil p\rceil$-perturbation determinant. This determinant is approximated via a finite element procedure on $B_R\setminus\overline U$ with explicit error bounds. A thresholding procedure then yields an approximation for the zero set of $\det_{\lceil p\rceil}(\mathrm{Id}_{L^2}+\mathcal K(k))$.
\subsection{Discussion}
The study of cavity resonances is much older than the study of quantum mechanical resonances. The foundational
work is generally ascribed to Helmholtz \cite{H1863}, who in the 1850s had constructed devices which were designed to 
identify special frequencies from within a sound wave. These {\em Helmholtz resonators} consisted of
a small aperture at one end to admit sound, and a larger opening at the other to emit it. 

The operator theoretic foundations of the study of cavity resonances are usually based on various approaches to the theory of 
scattering by obstacles. A good early treatment of both quantum-mechanical and obstacle scattering may be found in the seminal 
text of Newton \cite{Newton}. The book of  Lax and  Phillips \cite{LP} is perhaps the most famous work on scattering theory in a semi-abstract setting, though  the whole
 subject was extensively researched in the Soviet school by mathematicians such as Faddeev and Pavlov, see, e.g. \cite{FP},
 and also the monograph of Yafaev \cite{Y}.

There is extensive work  in the applied mathematics literature too, devoted to  estimating  resonances in
various asymptotic regimes, either geometrical or semiclassical: see \cite{HS,Helffer,Nedelec,PS,SZ} as starting
points. Assemblies of resonators are proposed as cloaking devices in a variety of contexts, see, e.g., \cite{Ammari}; the mathematical
treatment of acoustic waveguides is very similar to that of optical or quantum waveguides. A very up-to-date overview of the current state of the art in the subject may be found in \cite{DZ}.

Numerical approaches use a variety of techniques of which complex scaling (rediscovered as `perfectly matched layers'), boundary
integral techniques and combinations of these with special finite element methods are the most common.

Separately, recent years have seen a flurry of activity in research revolving around the $\SCI$ concept. In addition to \cite{Hansen11,AHS,Ben-Artzi2015a} which have been mentioned above, we point out \cite{Colbrook2019,Colbrook2019a} where some of the theory of spectral computations has been further developed; \cite{Rosler2019a}  where this has been applied to certain classes of unbounded operators; \cite{Becker2020b} where solutions of PDEs were considered; and \cite{Colbrook2019c} where the authors show how to perform certain spectral computations with error bounds.
\bigskip
\paragraph{\emph{Organization of the paper}}
In Section \ref{sec:DtN} we analyze the inner and outer DtN maps and obtain a new expression equivalent to \eqref{eq:Ker=0} in which $M_\text{inner}(k)+M_\text{outer}(k)$ is  replaced by an expression of the form $\mathrm{Id}_{L^2}$+perturbation.
Section \ref{sec:Matrix_Elements} is dedicated to the construction of a finite element method for the approximation of this perturbation, and in Section \ref{sec:Algorithm} the algorithm for approximating $\Res(U)$ is defined and shown to converge.
Finally, in Section \ref{sec:num} we provide some numerical examples.
\section{Formulas for the inner and outer DtN maps}\label{sec:DtN}
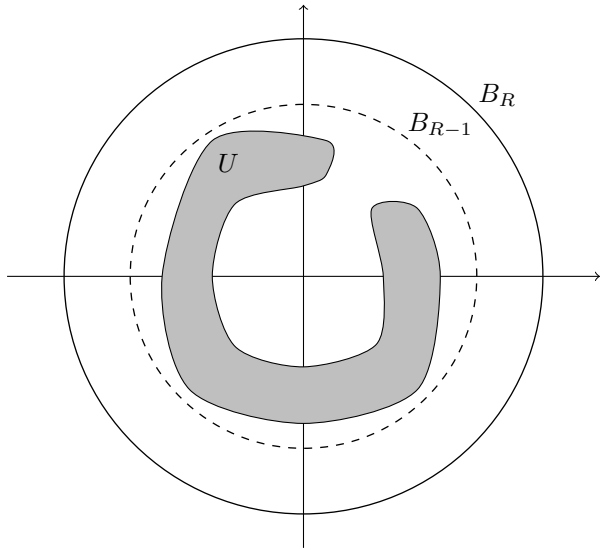
\begin{figure}
\centering
\begin{tikzpicture}[scale=0.3, line join=round]
	\draw[->] (-13,0) -- (13,0);
	\draw[->] (0,-12) -- (0,12);
	
	\filldraw[fill = gray!50] plot [smooth cycle, tension=0.5] coordinates {
		(-6.2,0) (-5,-5) (0,-6.5) (5,-5) (6,0) (5,3)
		(3,3) (3.5,0) (3.2,-3) (0,-4) (-3,-3.1) (-4,0) (-3,3.2) (0,4) (1,4.5) (1,6) (-4,6)
	};
	
	\draw (-3.3	,5) node{$U$};
	\draw[line width=0.5pt, dashed] (0,0) circle(7.6);
	\draw (6,6.7) node{$B_{R-1}$};
	\draw[line width=0.5pt] (0,0) circle(10.5);
	\draw (8.5,8) node{$B_R$};
\end{tikzpicture}
\caption{A sketch of the resonator $U$}
\end{figure}
As a first step, we prove a weaker version of Theorem \ref{th:mainth}. We show that under the additional assumption that the diameter of the  resonator is known \emph{a priori}, there exists an algorithm that computes the set of resonances in one limit. To this end we define
\begin{equation*}
	\Omega_R
	:=
	\left\{
		\emptyset\neq U\subset\R^2\text{ open}\,|\,\partial U\in\mathcal C^2,\,\overline{U}\subset B_{R-1}
	\right\}
\end{equation*}
where $B_\rho$ is the open ball about the origin of radius $\rho>0$. Then one has:
\begin{theorem}\label{th:fixed_R}
	There exists a family of arithmetic algorithms $\{\Gamma_n^R\}_{n\in\N}$ depending on $R$, such that $\Gamma_n^R(U)\to\Res(U)$ as $n\to+\infty$ for any $U\in\Omega_R$, where the convergence is in the sense of the Attouch-Wets metric. That is, $\SCI(\Omega_R,\Lambda,\Res(\cdot),(\mathrm{cl}(\C),d_\mathrm{AW}))=1$.
\end{theorem}
Henceforth, the radius $R>1$ will remain fixed until stated otherwise. 
In order to prove Theorem \ref{th:fixed_R}, we start by studying the DtN maps. The goal is to recast the formula \eqref{eq:Ker=0} into a form that can be implemented numerically. We introduce the orthonormal basis on $\partial B_R$
 	\begin{equation*}
	e_n(\theta):= \f{e^{\i n\theta}}{\sqrt{2\pi R}} ,\quad\theta\in[0,2\pi),\,n\in\Z,
	\end{equation*}
 which will be used frequently throughout the paper.
\subsection{The outer map}
The outer DtN map acts on a function $\phi\in H^1(\del B_R)$ as
\begin{align*}
	\Lout(k)\phi = \del_\nu u,
\end{align*}
where $u\in L^2(\R^2\setminus \overline{B}_R)$ solves the problem
\begin{align*}
\begin{cases}
	(-\Delta-k^2)u = 0 &\text{in }\R^2\setminus \overline{B}_R\\
	\hfill u = \phi &\text{on }\del B_R.
\end{cases}
\end{align*}
In the orthonormal basis $\left\{e_n\right\}_{n\in\Z}$ the map $\Lout(k)$ has the explicit representation
\begin{align*}
	\Lout(k) = \diag\left( -k\f{H_{|n|}^{(1)}{'}(kR)}{H_{|n|}^{(1)}(kR)} \,,\;n\in\Z \right),
\end{align*}
where $H_\nu^{(1)}$ denote the Hankel functions of the first kind.
Using well-known identities for Bessel functions (cf. \cite{Lebedev}), this can be rewritten as
\begin{equation}\label{eq:outer-dtn}
	\Lout(k) = \diag\left( \f{|n|}{R}-k\f{H_{|n|-1}^{(1)}(kR)}{H_{|n|}^{(1)}(kR)} \,,\;n\in\Z \right).
\end{equation}
Moreover, it can be seen that $\f{H_{|n|-1}^{(1)}(z)}{H_{|n|}^{(1)}(z)}\sim \f{z}{2|n|}$ as $|n|\to+\infty$, so that the matrix elements of $\Lout$ grow linearly in $|n|$.
\subsection{The inner map}
We first show that the inner DtN map can be decomposed into a part that does not depend on $U$ and a bounded part:
\begin{lemma}\label{lemma:Lambda_inner}
	Let $\Linull$ denote the free inner DtN map, i.e. the one with $U=\emptyset$. Then there exists a meromorphic operator-valued function  $\C\ni k\mapsto \mathcal{K}(k)$ with values in the bounded operators on $L^2(\del B_R)$ such that
\begin{equation}\label{eq:Inner_decomposition2}
	\Lin(k) = \Linull(k) + \mathcal{K}(k).
\end{equation}
\end{lemma}
\begin{proof}
For the inner DtN map, we  first isolate the contribution of the resonator using the following calculation. Let $\phi\in H^1(\del B_R)$. Then $\Lin(k)\phi = \del_\nu u|_{\del B_R}$, where $u$ solves
\begin{align}\label{eq:Lambda_inner_Def}
	\begin{cases}
		(-\Delta-k^2)u = 0 &\text{in }B_R\setminus\overline{U},\\
		 \hfill u = \phi &\text{on }\del B_R,\\
		 \hfill u = 0 &\text{on }\del U.
	\end{cases}
\end{align}
Define a new map  $S(k):H^1(\del B_R)\to H^{\f32}(B_R)$ by $S(k)\phi = w$, where $w$ solves
\begin{align*}
\begin{cases}
	(-\Delta-k^2)w = 0 &\text{ in }B_R,\\
	\hfill w = \phi &\text{ on }\del B_R.
\end{cases}
\end{align*}
That is, $S(k)\phi$ is the harmonic extension of $\phi$ into $B_R$. The regularity properties of $S(k)$ follow from \cite[Th. 4.21]{McLean}. Next choose any smooth radial function $\rho$ satisfying
\begin{align}\label{eq:rho_def}
	\begin{cases}
		\hfill \rho \equiv 1 & \text{ in a neighborhood of } \del B_R\\
		\hfill \rho \equiv 0 & \text{ in a neighborhood of }  B_{R-1}
	\end{cases}
\end{align}
(e.g. $\rho(r)$ can be chosen as a piecewise polynomial in $r$; then the values $\rho(r)$ can be computed from $r$ in finitely many algebraic steps) and introduce the function
\begin{align*}
	v:=u-\rho S(k)\phi
\end{align*}
on $B_R\setminus\overline U$. Then by construction we have
\begin{align*}
	\left\{
	\begin{array}{rll}
		(-\Delta-k^2)v  =&\hspace{-2mm} (2\nabla\rho\cdot\nabla + \Delta\rho) S(k)\phi &\text{ in }  B_R\setminus\overline{U},\\[1mm]
		 v =&\hspace{-2mm} 0  &\text{ on }\del B_R,\\
		 v =&\hspace{-2mm} 0  &\text{ on }\del U.
	\end{array}
	\right.
\end{align*}
In operator terms, this means that
\begin{align}\label{eq:H_D_def}
	v = (H_\mathrm{D}-k^2)^{-1}T_\rho S(k)\phi,
\end{align} 
where $\HD$ denotes the Laplacian on $L^2(B_R\setminus\overline U)$ with homogeneous Dirichlet  boundary condition on $\del (B_R\setminus\overline U)$ and where we have denoted	
	\[
	T_\rho:=2\nabla\rho\cdot\nabla +\Delta\rho.
	\]
Substituting the new representation for $v$ into $u=\rho S(k)\phi+v$, we immediately obtain
\begin{align}\label{eq:Inner_decomposition}
\begin{aligned}
	\Lin(k)\phi &= \del_\nu u|_{\del B_R}\\
	&= \del_\nu (\rho S(k)\phi) + \del_\nu (H_\mathrm{D}-k^2)^{-1}T_\rho S(k)\phi\\
	&= \underbrace{\del_\nu \rho}_{=0}S(k)\phi + \underbrace{\rho}_{=1}\del_\nu S(k)\phi + \del_\nu (H_\mathrm{D}-k^2)^{-1}T_\rho S(k)\phi\\
	&= \Linull(k)\phi + \del_\nu (H_\mathrm{D}-k^2)^{-1}T_\rho S(k)\phi.
\end{aligned}
\end{align}
Let us next study the operator $\del_\nu (H_\mathrm{D}-k^2)^{-1}T_\rho S(k)$ appearing in the last term in \eqref{eq:Inner_decomposition}. The harmonic extension operator $S(k)$ extends to a bounded operator
\begin{align*}
	S(k):L^2(\del B_R)\to H^{\f12}(B_R)
\end{align*}
(cf. \cite[Th. 6.12]{McLean}). Moreover, we have that $T_\rho:H^{\f12}(B_R)\to H^{-\f12}(B_R)$ is bounded. Under the assumption that $\del U$ is $\mathcal{C}^2$, we may conclude from elliptic regularity that 
\begin{align*}
	(H_\mathrm{D}-k^2)^{-1}T_\rho S(k) : L^2(\del B_R)\to H^{\f32}(B_R)
\end{align*}
as a bounded operator. Finally, by the trace theorem, the normal derivative operator $\del_\nu$ is continuous from $H^{\f32}(B_R)$ to $L^2(\del B_R)$, so we conclude that
\begin{align*}
	\del_\nu(H_\mathrm{D}-k^2)^{-1}T_\rho S(k):L^2(\del B_R)\to L^2(\del B_R)
\end{align*}
is bounded. This completes the proof.
\end{proof}
The decomposition in Lemma \ref{lemma:Lambda_inner} allows us to reduce the study of $\Lin$ to that of $\Linull$. Next, we note that $\Linull$ can in fact be represented as a diagonal operator in the basis $\left\{e_n \right\}_{n\in\Z}$ in a very similar manner to $\Lout$. Indeed, one has
\begin{align}
	\Linull(k) &= \diag\left( k\f{J_{|n|}'(kR)}{J_{|n|}(kR)}\,,\;n\in\Z \right)\label{eq:m-inner0}\\
	&= \diag\left( \f{|n|}{R}-k\f{J_{|n|+1}(kR)}{J_{|n|}(kR)}\,,\;n\in\Z \right),\notag
\end{align}
where $J_\nu$ denote the Bessel functions of the first kind. It can be seen that $\f{J_{|n|+1}(z)}{J_{|n|}(z)}\sim \f{z}{2{|n|}}$ as ${|n|}\to+\infty$, so that the matrix elements of $\Linull$ grow linearly in ${|n|}$.
\subsection{Approximation procedure}\label{sec:approximation_procedure}
From equations \eqref{eq:outer-dtn},  \eqref{eq:Inner_decomposition2} and \eqref{eq:m-inner0} we have
\begin{align*}
	\Lin(k)+\Lout(k) = \f{2}{R}N+\BH(k)+\BJ(k)+\mathcal{K}(k),
\end{align*}
where 
\begin{align*}
	\BJ(k) &= \diag\left(-k\f{J_{{|n|}+1}(kR)}{J_{|n|}(kR)}\,,\;n\in\Z \right)\\
	\BH(k) &= \diag\biggl(-k\f{H_{{|n|}-1}^{(1)}(kR)}{H_{|n|}^{(1)}(kR)} \,,\;n\in\Z \biggr)\\
	N &= \diag\big({|n|}\,,\;n\in\Z\big)\\
	\mathcal{K}(k)&= \del_\nu(H_\mathrm{D}-k^2)^{-1}(2\nabla\rho\cdot\nabla+\Delta\rho) S(k).
\end{align*}
Note that $\BH,\BJ,\mathcal{K}$ are bounded and analytic in $k$. In order to determine whether $\ker(\Lin+\Lout)\neq\{0\}$ we want to transform $\Lin+\Lout$ into an operator of the form $I+(\text{compact})$. To this end, we introduce the bijective operator $\mathcal N := N + P_0 = \diag(...,3,2,1,1,1,2,3,...)$, where $P_0$ denotes the projection onto the zeroth component.
This new notation brings $\Lin+\Lout$ into the form
\begin{align*}
	\Lin(k)+\Lout(k) &= \f 2R\mathcal N+\BH(k)+\BJ(k)+\mathcal K(k)-\f 2R P_0\\
	&= \f 2R \mathcal  N^{\f12}\left(I+\f R2 \mathcal N^{-\f12}\big(\BH(k)+\BJ(k)+\mathcal K(k)\big)\mathcal N^{-\f12}-\mathcal N^{-\f12}P_0\mathcal N^{-\f12}\right)\mathcal N^{\f12}\\
	&= \f 2R \mathcal N^{\f12}\left(I+\f R2 \mathcal N^{-\f12}\big(\BH(k)+\BJ(k)+\mathcal K(k)\big)\mathcal N^{-\f12}-P_0\right)\mathcal N^{\f12}.
\end{align*}
Because $\mathcal N^\f12 = \diag(...,\sqrt{3},\sqrt 2, 1,1,1,\sqrt 2,\sqrt 3,...)$ is bijective from its domain to $L^2(\del B_R)$, we have
\begin{align*}
	\ker(\Lin+\Lout) = \{0\} \quad\Leftrightarrow\quad \ker\left(I+\f R2 \mathcal N^{-\f12}\big(\BH(k)+\BJ(k)+\mathcal K(k)\big)\mathcal N^{-\f12}-P_0\right) = \{0\}.
\end{align*}
Now, observe that for all $p>2$, $\f R2 \mathcal N^{-\f12}\big(\BH(k)+\BJ(k)+\mathcal K(k)\big)\mathcal N^{-\f12}-P_0\in C_p$, the $p$-Schatten class (because $\BH,\BJ,\mathcal K$ are bounded and $\mathcal N^{-\f12}$ is in $C_p$). Therefore, its perturbation determinant $\det_{\lceil p\rceil}$ is defined, where $\lceil\cdot\rceil$ denotes the ceiling function (cf. \cite[Sec. XI.9]{DS2}).
Our task is reduced to finding the zeros $k\in\C^-$ of the analytic function
\begin{align*}
	\text{det}_{\lceil p\rceil}\left(I+\f R2 \mathcal N^{-\f12}\big(\BH(k)+\BJ(k)+\mathcal K(k)\big)\mathcal N^{-\f12}-P_0\right).
\end{align*}
In order to approximate this determinant by something computable, we first prove that a square truncation of the operator matrix will converge in Schatten norm. 
\begin{lemma}\label{lemma:abstract_projection_estimate}
	Let $k\in\C^-$ and for $n\in\N$ let $P_n:L^2(\del B_R)\to \operatorname{span}\{e_{-n},\dots e_n\}$ be the orthogonal projection. Then there exists a constant $C>0$ depending only on the set $U$ such that
	\begin{align*}
		\left\| \mathcal N^{-\f12}(\BH+\BJ+\mathcal K)\mathcal N^{-\f12} - P_n\mathcal N^{-\f12}(\BH+\BJ+\mathcal K)\mathcal N^{-\f12}P_n \right\|_{C_p}\leq Cn^{-\f12+\f1p}.
	\end{align*}
\end{lemma}
\begin{proof}
	The proof is given by a  direct calculation:
	\begin{align*}
		\Big\| \mathcal N^{-\f12}(\BH+\BJ  +\mathcal K)\mathcal N^{-\f12} &-P_n\mathcal N^{-\f12}(\BH+\BJ+\mathcal K)\mathcal N^{-\f12}P_n \Big\|_{C_p} \\
		&\leq \left\| (I-P_n)\mathcal N^{-\f12}(\BH+\BJ+\mathcal K)\mathcal N^{-\f12}\right\|_{C_p} \\
		&\qquad \;\,+ \left\|P_n\mathcal N^{-\f12}(\BH+\BJ+\mathcal K)\mathcal N^{-\f12}(I-P_n) \right\|_{C_p}\\
		&\leq 2\|\BH+\BJ+\mathcal K\|_{\mathcal L(L^2(\del B_R))}\left\|(I-P_n)\mathcal N^{-\f12}\right\|_{C_p} \\
		&\leq C\|\BH+\BJ+\mathcal K\|_{\mathcal L(L^2(\del B_R))}n^{\f{1-\f p2}{p}}.
	\end{align*}
\end{proof}
Our problem is therefore reduced to computing the matrix elements of the truncated operator $P_n\mathcal N^{-\f12}(\BH+\BJ+\mathcal K)\mathcal N^{-\f12}P_n-P_0$. This is done by performing a  finite element procedure on $B_R\setminus\overline U$. This is the objective of the next section.
\begin{remark}
We remind the reader that the algorithm is assumed to have access to the values of Bessel and Hankel functions (recall \eqref{eq:evaluation_set}) so that the values of $P_n\mathcal N^{-\f12}(\BH+\BJ)\mathcal N^{-\f12}P_n-P_0$ can be computed in finitely arithmetic operations. Hence we only need to compute $P_n\mathcal N^{-\f12}\mathcal{K}\mathcal N^{-\f12}P_n$.
\end{remark}
\section{Matrix elements of the inner DtN map}\label{sec:Matrix_Elements}
The goal of this section is to compute the matrix elements
\begin{align*}
	\mathcal K_{\alpha\beta} := \int_{\del B_R} \overline{e_\beta} \mathcal K(k)e_\alpha\,d\sigma
\end{align*}
where $\{e_\alpha\}_{\alpha\in\Z}$ is the orthonormal basis of $L^2(\partial B_R)$ defined in the previous section and where we recall that $\mathcal{K}(k)= \del_\nu(H_\mathrm{D}-k^2)^{-1}T_\rho S(k)$.
This is done by converting the integral on the boundary $\partial B_R$ to an integral on the interior, using Green's theorem. To this end, we extend the basis functions to the interior, by defining (in polar coordinates)
\begin{align*}
	E_\beta(r,\theta) := \rho(r)e_\beta(\theta),
\end{align*}
where $\rho$ was defined in \eqref{eq:rho_def}, and denote
\begin{align}
	v_\alpha &:= (H_\mathrm{D}-k^2)^{-1}T_\rho S(k) e_\alpha\notag\\
	f_\alpha &:= T_\rho S(k) e_\alpha\label{eq:f-alpha}
\end{align}
so that $v_\alpha = (H_\mathrm{D}-k^2)^{-1}f_\alpha$. Now note that by Green's theorem
\begin{align}
	\mathcal K_{\alpha\beta} &= \int_{\del B_R} \overline {e_\beta}\del_\nu v_\alpha\,d\sigma \nonumber\\
	&= \int_{B_R\setminus\overline U} \overline {E_\beta}\Delta v_\alpha\,dx + \int_{B_R\setminus\overline U} \nabla \overline {E_\beta}\cdot\nabla v_\alpha\,dx \nonumber \\
	&= \int_{B_R\setminus\overline U} \overline {E_\beta}(-f_\alpha-k^2v_\alpha)\,dx + \int_{B_R\setminus\overline U} \nabla \overline {E_\beta}\cdot\nabla v_\alpha\,dx \nonumber \\
	&=  \int_{B_R\setminus\overline U} \nabla \overline {E_\beta}\cdot\nabla v_\alpha\,dx  - k^2\int_{B_R\setminus\overline U} \overline {E_\beta}v_\alpha\,dx - \int_{B_R\setminus\overline U} \overline {E_\beta}f_\alpha\,dx. \label{eq:matrix_elements}
\end{align}
Next, note that the integrand in the third term on the right hand side is explicitly given. Indeed, one can easily see that 
\begin{align}\label{eq:explicit_extension}
	S(k)e_\alpha(r,\theta) = \f{J_\alpha(kr)}{J_\alpha(kR)} \f{e^{i\alpha\theta}}{\sqrt{2\pi R}},
\end{align}
where $J_\alpha$ are the Bessel functions of the first kind. Hence, the last integral in \eqref{eq:matrix_elements} can be explicitly approximated by quadrature, with a known bound for the error. In order to compute the first and second terms in the right hand side of \eqref{eq:matrix_elements}, we need to approximate the function $v_\alpha$, which is the solution of the boundary value problem
\begin{align}\label{eq:v_m-problem}
	\begin{cases}
		(-\Delta-k^2)v_\alpha = (2\nabla\rho\cdot\nabla + \Delta\rho)S(k)e_\alpha &\text{in }B_R\setminus \overline U,\\
		\phantom{(-\Delta-k^2)} v_\alpha = 0 &\text{on }\del(B_R\setminus \overline U).
	\end{cases}
\end{align}
Note again that the right hand side of this equation is explicitly given in terms of \eqref{eq:explicit_extension}. 
\subsection{Finite element approximation of $\boldsymbol v_\alpha$}\label{sec:FEM}
For notational convenience we write \[\cO:=B_R\setminus\overline U.\]
In the current subsection we prove: 
\begin{prop}\label{prop:FEM_estimate}
	Let $v_\alpha$ be the solution of \eqref{eq:v_m-problem}. For small discretization parameter $h>0$ there exists a piecewise linear function $v_\alpha^h$ which is computable from a finite subset of \eqref{eq:evaluation_set} in finitely many algebraic steps, which satisfies the error estimate
	\begin{align*}
		\|v_\alpha-v_\alpha^h\|_{H^1(\cO)} &\leq C h^{\f13} \|f_\alpha\|_{H^1(\cO)},
	\end{align*}
	where $C$ is independent of $h$ and $\alpha$, and where $f_\alpha$ is defined in \eqref{eq:f-alpha}.
\end{prop}
The reader should bear in mind that all constants obtained in our estimates will actually depend on $k$. However, these constants are uniformly bounded, as long as $k^2$ varies in a compact set separated from the spectrum of $H_\mathrm{D}$ (recall that $\HD$ is the Laplacian on $L^2(\cO)$ with homogeneous Dirichlet  boundary conditions on $\del \cO$). Note that, since $\sigma(H_\mathrm{D})$ is strictly positive and discrete, any compact subset of $\C^-$ has this property.

Let $h>0$ be our discretization parameter. We start by defining what we mean by the \emph{neighbors} of a grid point $j\in h\Z^2$.
\begin{de}
	Let $j=(j_1,j_2)\in h\Z^2$. The set of \emph{neighbors} of $j$ is the set of all grid points contained in the cell $[j_1-h,j_1+2h]\times [j_2-h,j_2+2h]$ (cf. Figure \ref{fig:cell}).
\end{de}
\begin{remark}
The apparent asymmetry in the definition above is merely due to the fact that a point $j$ defines a cell that is ``northeast'' of it, i.e. $[j_1,j_1+h]\times[j_2,j_2+h]$.
\end{remark}
Define a uniform rectangular grid by 
$$
	L_h:=\left\{j\in h\Z^2\cap B_R	\,\middle|\,\text{all neighbors of $j$ are contained in $\cO$} \right\}
$$
and note that $L_h$ can be completely constructed from the knowledge of $\chi_U(j),\; j\in h\Z^2\cap B_R$ in finitely many steps. Next, define the open polygonal domain $\cO_h$ generated by $L_h$ as follows.
\begin{align*}
	\cO_h:=\operatorname{int}\bigg(\bigcup_{j\in L_h} [0,h]^2+j\bigg).
\end{align*}
Note that as soon as $h<\f{1}{\max(\text{curv}(\del\cO))}=\f{1}{\max(\text{curv}(\del U))}$ one will have $\cO_h\subset \cO$ and for all $x\in\del\cO_h$  (cf. Figure \ref{fig:cell})
\begin{align}\label{eq:boundary_width}
	\f{\sqrt 3}{2}h\leq \dist(x,\del\cO) \leq 3\sqrt{2}h.
\end{align}
\begin{figure}[H]
	\centering
%
%

\begin{tikzpicture}[>=stealth, scale=1.5]
	\clip (-2.8,-1.8) rectangle (4.8,2.8);
	
	\filldraw[dashed, color=black, fill=white] (0.5,2.8) circle(0.9);
	\draw[->, color=black] (0.5,2.8) -- (-0.32,2.43);
	\draw (4.3,1.2) node{$\partial\mathcal O$};
	\draw (0.15,2.53) node{\textcolor{black}{$h$}};

	\draw[black, line width=1.5pt, domain=-3:6, samples=300, shift={(-0.13,0.06)}] 
		plot (\x,{-0.2*sin(pi*46*\x) - (\x/4)^5+2.03}) (0, 2.7);
	
	\foreach \x in {-3,-2,-1,0,1,2}{
		\fill[color=blue, opacity=0.15] (\x,-1.8) rectangle (\x+1,-1);
		\fill[color=blue, opacity=0.15] (\x,-1) rectangle (\x+1,0);
	}
	\fill[color=blue, opacity=0.15] (0,0) rectangle (1,1);
	
	\foreach \y in {-2,-1,0,1,2}{
		\draw[color=black] (-2.8,\y) -- (4.8,\y);
	}
	\foreach \x in {-2,-1,2,3,4,5}{
		\draw[color=black] (\x,-1.8) -- (\x,2.8);
	}
	\draw[color=black] (1,-1.8) -- (1,2);
	\draw[color=black] (0,-1.8) -- (0,2);
	
	\foreach \y in {-1,0,1,2}{
		\foreach \x in {-2,-1,0,1,2,3,4}{
			\filldraw[color=gray!10!black, fill=white] (\x,\y) circle(0.04);
		}
	}
	
	\draw (0.13,-0.13) node{$j$};
	\draw (-1.5,-1.32) node{$h$};
	\draw[<->] (-2,-1.45) -- (-1,-1.45);
	\draw (-1.6,-0.6) node{$\mathcal O_h$};
	
	\foreach \y in {-1,0,1,2}{
		\foreach \x in {-1,0,1,2}{
			\fill[black] (\x,\y) circle(0.05);
		}
	}
	
	\draw (0.1,1.4) node[fill=white, inner sep=1pt]{\small{\color{red}$\geq \frac{\sqrt{3}}{2} h$}};
	\draw[<->, color=red] (0.5,1) -- (0.5,1.87);
	
\end{tikzpicture}

	\caption{Illustration of  \eqref{eq:boundary_width}. The point $j$ is included in $L_h$, if and only if all the solid black points are in $\cO=B_R\setminus\overline U$. These solid black points are the neighbors of $j$.}
	\label{fig:cell}
\end{figure}
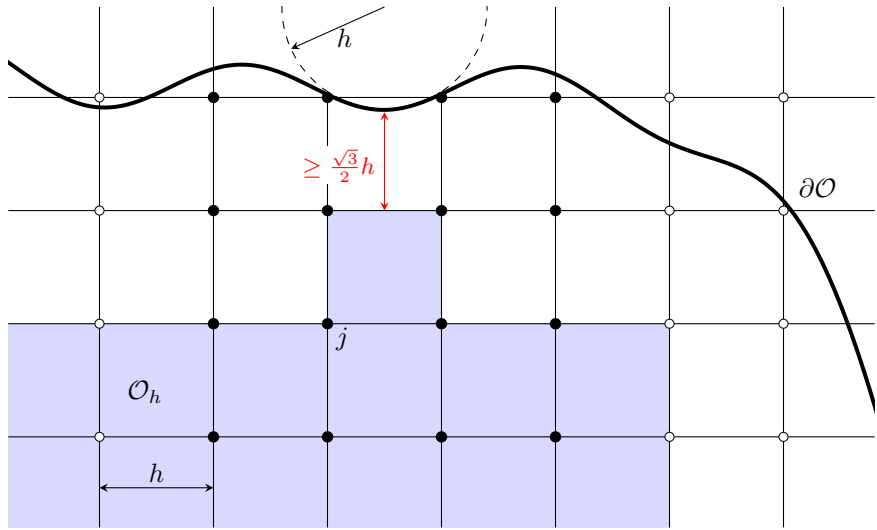
Finally, we choose the obvious triangulation $\mathcal T_h$ for $\cO_h$ obtained by splitting each square $[0,h]^2+j$ along its diagonal and denote by $W^h$ the corresponding P1 finite element space with zero boundary conditions, i.e.
\begin{align*}
	W^h:= \left\{ w\in H^1(\cO)\,\middle|\, w\equiv 0 \text{ in } \cO\setminus\cO_h \text{ and $w$ piecewise linear in }\cO_h \text{ w.r.t. } \mathcal{T}_h \right\}
\end{align*}
Next, we will construct a finite element approximation for the solution of problem \eqref{eq:v_m-problem} starting from C\'ea's classical lemma:
\begin{lemma}[C\'ea's Lemma, {\cite[Th. 2.4.1]{Ciarlet}}]\label{lemma:Cea}
	Let $\mathcal W^h\subset H^1_0(\cO)$ be a family of subspaces and assume the bilinear form $a:H^1_0(\cO)\times H^1_0(\cO)\to\C$ and linear functional $f\in H^{-1}(\cO)$ satisfy the hypotheses of the Lax-Milgram theorem. Let $v$ be the solution of the problem
	\begin{align*}
		a(v,w) = \langle f,w\rangle\quad\text{ for all }w\in H^1_0(\cO)
	\end{align*}
	and $v^h$ be the solution of 
	\begin{align}\label{eq:FEM_problem}
		a(v^h,w^h) = \langle f,w^h\rangle\quad\text{ for all }w^h\in \mathcal W^h.
	\end{align}
	Then there exists $C>0$ s.t.
	\begin{align}\label{eq:Cea-est}
		\|v-v^h\|_{H^1(\cO)} \leq C \inf_{w^h\in \mathcal W^h} \|v-w^h\|_{H^1(\cO)}
	\end{align}
	for all $h>0$.
\end{lemma}
We apply this lemma to our present situation by choosing
\begin{align*}
\begin{cases}
	\hfill \mathcal W^h &= \quad W^h\\
	a(v,w)              &= \quad\int_{\cO} \nabla v\cdot\nabla w\,dx\\
	\hfill f            &=\quad f_\alpha\\
	\hfill v            &=\quad v_\alpha \quad\text{ as in \eqref{eq:v_m-problem}}\\
	\hfill v_\alpha^h        &=\quad\text{the solution of \eqref{eq:FEM_problem}}.
\end{cases}
\end{align*}
\begin{remark}\label{remark:numint1}
	Note that the finite element approximation $v^h_\alpha$ is \emph{not} computable from the information in \eqref{eq:evaluation_set} in finitely many algebraic operations. An additional step of numerical integration is necessary to achieve this. We will revisit this issue at the end of the section.
\end{remark}
Lemma \ref{lemma:Cea} tells us that in order to estimate the approximation error $\|v_\alpha-v_\alpha^h\|_{H^1(\cO)}$ it is enough to estimate the \emph{projection error} $\inf_{w^h\in W^h} \|v_\alpha-w^h\|_{H^1(\cO)}$. For the latter, we have
\begin{align*}
	\|v_\alpha-w^h\|_{H^1(\cO)} &\leq 
	\|v_\alpha\|_{H^1(\cO\setminus\cO_h)} + \|v_\alpha-w^h\|_{H^1(\cO_h)} 
\end{align*}
for any $w^h\in W^h$. To estimate the right hand side, we further decompose $\cO_h$ into a boundary layer and an interior part. More precisely, we write
\begin{align*}
	\cO_h = \cO_h^\text{int}\cup\cO_h^\text{bdry},
\end{align*}
where $\cO_h^\text{bdry}$ is the union of all cells having nonempty intersection with $\del\cO_h$, and $\cO_h^\text{int}=\cO_h\setminus \cO_h^\text{bdry}$ (cf. Figure \ref{fig:O_h}). 
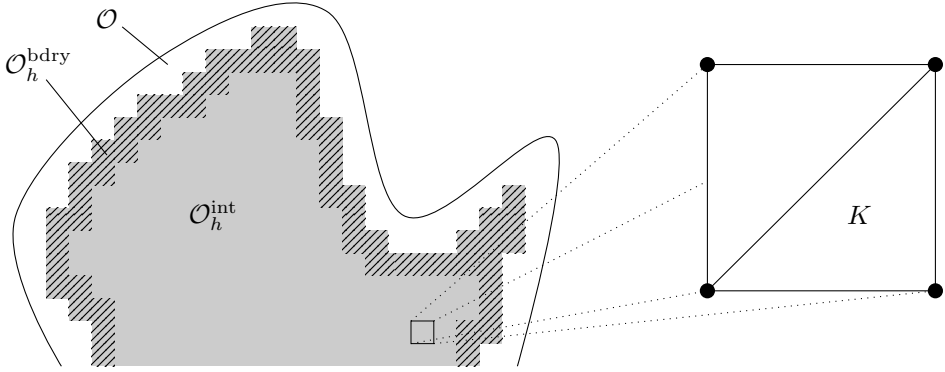
\begin{figure}[H]
	\centering
%

\begin{tikzpicture}[>=stealth]

	\draw[black] plot [smooth, tension=0.6] coordinates { (0.5,-2) (-0.1,0) (1.8,2.1) (4,2.7) (5,0) (7,1) (6.5,-2)};
	
	\def \q{0.3}
	
	\foreach \x in {3,...,19}{
		\fill[fill = gray!40] (\q*\x,-2) rectangle (\q*\x+\q,-2+\q);
	}
	\foreach \x in {3,...,20}{
		\fill[fill = gray!40] (\q*\x,-2+\q) rectangle (\q*\x+\q,-2+\q+\q);
	}
	\foreach \x in {2,...,20}{
		\fill[fill = gray!40] (\q*\x,-2+2*\q) rectangle (\q*\x+\q,-2+\q+2*\q);
	}
	\foreach \x in {1,...,20}{
		\fill[fill = gray!40] (\q*\x,-2+3*\q) rectangle (\q*\x+\q,-2+\q+3*\q);
	}
	\foreach \x in {1,...,20}{
		\fill[fill = gray!40] (\q*\x,-2+4*\q) rectangle (\q*\x+\q,-2+\q+4*\q);
	}
	\foreach \x in {1,...,15,19,20,21}{
		\fill[fill = gray!40] (\q*\x,-2+5*\q) rectangle (\q*\x+\q,-2+\q+5*\q);
	}
	\foreach \x in {1,...,14,20,21}{
		\fill[fill = gray!40] (\q*\x,-2+6*\q) rectangle (\q*\x+\q,-2+\q+6*\q);
	}
	\foreach \x in {2,...,14,21}{
		\fill[fill = gray!40] (\q*\x,-2+7*\q) rectangle (\q*\x+\q,-2+\q+7*\q);
	}
	\foreach \x in {2,...,13}{
		\fill[fill = gray!40] (\q*\x,-2+8*\q) rectangle (\q*\x+\q,-2+\q+8*\q);
	}
	\foreach \x in {3,...,13}{
		\fill[fill = gray!40] (\q*\x,-2+9*\q) rectangle (\q*\x+\q,-2+\q+9*\q);
	}
	\foreach \x in {4,...,13}{
		\fill[fill = gray!40] (\q*\x,-2+10*\q) rectangle (\q*\x+\q,-2+\q+10*\q);
	}
	\foreach \x in {5,...,12}{
		\fill[fill = gray!40] (\q*\x,-2+11*\q) rectangle (\q*\x+\q,-2+\q+11*\q);
	}
	\foreach \x in {7,...,12}{
		\fill[fill = gray!40] (\q*\x,-2+12*\q) rectangle (\q*\x+\q,-2+\q+12*\q);
	}
	\foreach \x in {8,...,12}{
		\fill[fill = gray!40] (\q*\x,-2+13*\q) rectangle (\q*\x+\q,-2+\q+13*\q);
	}
	\foreach \x in {10,11}{
		\fill[fill = gray!40] (\q*\x,-2+14*\q) rectangle (\q*\x+\q,-2+\q+14*\q);
	}
	
	\foreach \x in {3,19}{
		\fill[pattern=north east lines] (\q*\x,-2) rectangle (\q*\x+\q,-2+\q);
	}
	\foreach \x in {3,19,20}{
		\fill[pattern=north east lines] (\q*\x,-2+\q) rectangle (\q*\x+\q,-2+\q+\q);
	}
	\foreach \x in {2,3,20}{
		\fill[pattern=north east lines] (\q*\x,-2+2*\q) rectangle (\q*\x+\q,-2+\q+2*\q);
	}
	\foreach \x in {1,2,20}{
		\fill[pattern=north east lines] (\q*\x,-2+3*\q) rectangle (\q*\x+\q,-2+\q+3*\q);
	}
	\foreach \x in {1,15,16,17,18,19,20}{
		\fill[pattern=north east lines] (\q*\x,-2+4*\q) rectangle (\q*\x+\q,-2+\q+4*\q);
	}
	\foreach \x in {1,14,15,19,20,21}{
		\fill[pattern=north east lines] (\q*\x,-2+5*\q) rectangle (\q*\x+\q,-2+\q+5*\q);
	}
	\foreach \x in {1,2,14,20,21}{
		\fill[pattern=north east lines] (\q*\x,-2+6*\q) rectangle (\q*\x+\q,-2+\q+6*\q);
	}
	\foreach \x in {2,13,14,21}{
		\fill[pattern=north east lines] (\q*\x,-2+7*\q) rectangle (\q*\x+\q,-2+\q+7*\q);
	}
	\foreach \x in {2,3,13}{
		\fill[pattern=north east lines] (\q*\x,-2+8*\q) rectangle (\q*\x+\q,-2+\q+8*\q);
	}
	\foreach \x in {3,4,13}{
		\fill[pattern=north east lines] (\q*\x,-2+9*\q) rectangle (\q*\x+\q,-2+\q+9*\q);
	}
	\foreach \x in {4,5,12,13}{
		\fill[pattern=north east lines] (\q*\x,-2+10*\q) rectangle (\q*\x+\q,-2+\q+10*\q);
	}
	\foreach \x in {5,6,7,12}{
		\fill[pattern=north east lines] (\q*\x,-2+11*\q) rectangle (\q*\x+\q,-2+\q+11*\q);
	}
	\foreach \x in {7,8,12}{
		\fill[pattern=north east lines] (\q*\x,-2+12*\q) rectangle (\q*\x+\q,-2+\q+12*\q);
	}
	\foreach \x in {8,9,10,11,12}{
		\fill[pattern=north east lines] (\q*\x,-2+13*\q) rectangle (\q*\x+\q,-2+\q+13*\q);
	}
	\foreach \x in {10,11}{
		\fill[pattern=north east lines] (\q*\x,-2+14*\q) rectangle (\q*\x+\q,-2+\q+14*\q);
	}
		
	\draw (1.3,2.5) -- (1.9,2);
	\draw (0.3,1.8) -- (1.1,0.8);
	
	\draw (1.1,2.6)node[]{$\mathcal O$} ;
	\draw (0.2,2)node[]{$\mathcal O_h^{\mathrm{bdry}}$} ;
	\draw (2.5,0)node[]{$\mathcal O_h^{\mathrm{int}}$} ;

	\draw[dotted] (0.3*17,-2+0.3) -- (9, -1);
	\draw[dotted] (0.3*17,-2+0.6) -- (9, 2);
	\draw[dotted] (0.3*17+0.3,-2+0.3) -- (12, -1);
	\draw[dotted] (0.3*17+0.3,-2+0.6) -- (12, 2);
	\draw[line width=0.2] (0.3*17,-2+0.3) rectangle (0.3*17+0.3,-2+0.6);
	
	\filldraw[fill=white] (9, -1) rectangle (12, 2);
	\fill (9,-1) circle(0.1);
	\fill (12,-1) circle(0.1);
	\fill (9,2) circle(0.1);
	\fill (12,2) circle(0.1);
	\draw (9, -1) -- (12, 2);
	
	\draw (11,0) node[]{$K$} ;

\end{tikzpicture}

	\caption{Left: Sketch of $\cO=B_R\setminus\overline{U}$, $\cO_h$ and $\cO_h^{\mathrm{bdry}}$. Right: Sketch of the chosen triangulation.}
	\label{fig:O_h}
\end{figure}
Hence, for any $w^h\in W^h$ we have
\begin{align}\label{eq:FEM_error}
	\|v_\alpha-w^h\|_{H^1(\cO)} &\leq 
	\|v_\alpha-w^h\|_{H^1(\cO\setminus\cO_h^\text{int})} + \|v_\alpha-w^h\|_{H^1(\cO_h^\text{int})}  
\end{align}
Next, estimate the all three terms  on the r.h.s. of \eqref{eq:FEM_error} individually. To  estimate  the two last terms we will choose a particular ``test function'' $u^h\in W^h$.
We first   recall a classical interpolation estimate for linear finite elements  \cite[Th. 3.1.6]{Ciarlet}:
\begin{theorem}[Classical interpolation estimate]\label{th:interpolation_estimate}
	Let $K$ be an element of $\mathcal T_h$ and let $\Pi^h_K$ denote the linear interpolation operator on $H^2(K)$, i.e. $\Pi^h_Kw$ is the affine function with $(\Pi^h_Kw)(j)=w(j)$ for all corners $j$ of $K$. Then there exists $C>0$ (independent of $h,w,K$) such that
	\begin{align*}
		\|w - \Pi_K^hw\|_{H^1(K)} &\leq C h \|D^2w\|_{L^2(K)},\\
		\|w - \Pi_K^hw\|_{L^2(K)} &\leq C h^2 \|D^2w\|_{L^2(K)}.
	\end{align*}
\end{theorem}
Now, we choose an explicit function $u^h\in W^h$ as follows. 
\begin{itemize}
	\item For each finite element $K\subset\cO_h^{\mathrm{int}}$, let $u^h|_{K}=\Pi_K^h(v_\alpha)$ be the linear interpolation of $v_\alpha$, i.e. $u^h(j)=v_\alpha(j)$ for all nodes $j\in\del K\cap L_h$. 
	\item For boundary elements $K\subset\cO_h^{\mathrm{bdry}}$, we set
	\begin{align*}
		u^h(i) &= 0, \qquad\quad\text{for boundary nodes }i,\\
		u^h(j) &= v_\alpha(j), \quad\,\text{for interior nodes }j.
	\end{align*}
	Here a \emph{boundary node} is any $i\in L_h\cap\del\cO_h$. Note that testing whether a node is a boundary node can be done in finitely many steps, because boundary nodes are precisely those nodes whose neighbors are not all contained in $\cO$.
\end{itemize}
This defines a continuous, piecewise linear function in $W^h$. We have the following error estimates.
\begin{lemma}[third term in \eqref{eq:FEM_error}]
\label{lemma:Oh^int}
	One has the interior error estimate
	\begin{align*}
		\|v_\alpha-u^h\|_{H^1(\cO_h^{\mathrm{int}})} \leq C h \|v_\alpha\|_{H^2(\cO_h^{\mathrm{int}})}.
	\end{align*}
\end{lemma}
\begin{proof}
	This follows immediately from Theorem \ref{th:interpolation_estimate}.
\end{proof}
\begin{lemma}[first term in \eqref{eq:FEM_error}]\label{lemma:Oh^bdry}
	For any $q\in(2,\infty)$ there exists $C>0$ such that the boundary layer error estimate
	\begin{align*}
		\|v_\alpha-u^h\|_{H^1(\cO\setminus\cO_h^\text{int})} \leq C h^{\f12 - \f1q} \|v_\alpha\|_{H^2(\cO)}.
	\end{align*}
	holds.
\end{lemma}
\begin{proof}
	We begin with $\|v_\alpha-u^h\|_{L^2(\cO\setminus\cO_h^{\mathrm{int}})}$. Let $x\in \cO\setminus\cO_h^{\mathrm{int}}$ 	and note that similarly to eq. \eqref{eq:boundary_width} there exists $y\in \del\cO$ with $|x-y|\leq 4\sqrt 2 h$. By Morrey's inequality (cf. (28) in the proof of \cite[Th. 9.12]{Brezis}) this implies
	\begin{align}
		|v_\alpha(x)|  &= |v_\alpha(x)-v_\alpha(y)| \nonumber\\
		&\leq C|x-y|\|v_\alpha\|_{H^2(\cO\setminus\cO_h^{\mathrm{int}})} \nonumber \\
		&\leq Ch\|v_\alpha\|_{H^2(\cO\setminus\cO_h^{\mathrm{int}})},\label{eq:O^bdry_estimate}
	\end{align}
	Moreover, we note that by the definition of $u^h$, one has $\|u^h\|_{L^\infty(\cO\setminus\cO_h^{\mathrm{int}})} \leq \|v_\alpha\|_{L^\infty(\cO\setminus\cO_h^{\mathrm{int}})}$.
	These facts allow us to estimate
	\begin{align}
		\|v_\alpha-u^h\|_{L^2(\cO\setminus\cO_h^{\mathrm{int}})} &\leq \|v_\alpha\|_{L^2(\cO\setminus\cO_h^{\mathrm{int}})} +\|u^h\|_{L^2(\cO\setminus\cO_h^{\mathrm{int}})}  \nonumber\\
		&\leq |\cO\setminus\cO_h^{\mathrm{int}}|^{\f12}\big( \|v_\alpha\|_{L^\infty(\cO\setminus\cO_h^{\mathrm{int}})} +\|u^h\|_{L^\infty(\cO\setminus\cO_h^{\mathrm{int}})} \big) \nonumber\\
		&\leq Ch^{\f12}\|v_\alpha\|_{L^\infty(\cO\setminus\cO_h^{\mathrm{int}})} \nonumber\\
		&\leq Ch^{\f32}\|v_\alpha\|_{H^2(\cO\setminus\cO_h^{\mathrm{int}})} \label{eq:L2_estimate}
	\end{align}
	where we have used \eqref{eq:O^bdry_estimate} in the last line.
	
	Turning to the gradient, a similar estimate is obtained as follows. By definition of $u^h$, one has 
	\begin{align*}
		|\nabla u^h(x)| \leq
		\begin{cases}
 			0 & \text{ for } x\in\cO\setminus\cO_h \\
 			\f Ch \|v_\alpha\|_{L^\infty(K)} & \text{ for } x\in K\subset\cO_h^\mathrm{bdry}
		\end{cases}
	\end{align*}
	In particular, for every $K\subset \cO_h^\mathrm{bdry}$ one has 
	\begin{align}
		\|\nabla u^h\|_{L^2(K)} &\leq C\|v_\alpha\|_{L^\infty(K)} \nonumber\\
		&\leq C h \|v_\alpha\|_{H^2(K)} \qquad \text{by \eqref{eq:O^bdry_estimate}.}\label{eq:grad_sobolev_estimate}
	\end{align}
	Thus, using the fact that $H^1(\cO)\hookrightarrow L^q(\cO)$ for all $2<q<\infty$, we obtain
	\begin{align*}
		\tfrac12 \|\nabla v_\alpha-\nabla u^h\|^2_{L^2(\cO\setminus\cO_h^{\mathrm{int}})} 
		&\leq \|\nabla v_\alpha\|_{L^2(\cO\setminus\cO_h^{\mathrm{int}})}^2 +\|\nabla u^h\|_{L^2(\cO\setminus\cO_h^{\mathrm{int}})}^2  \\
		&= \|\nabla v_\alpha\|_{L^2(\cO\setminus\cO_h^{\mathrm{int}})}^2 +\|\nabla u^h\|_{L^2(\cO_h^{\mathrm{bdry}})}^2  \\
		&= \|\nabla v_\alpha\|_{L^2(\cO\setminus\cO_h^{\mathrm{int}})}^2 + \sum_{K\subset\cO_h^{\mathrm{bdry}}} \|\nabla u^h\|_{L^2(K)}^2\\
		&\leq |\cO\setminus\cO_h^{\mathrm{int}}|^{1 - \f2q}\|\nabla v_\alpha\|_{L^q(\cO\setminus\cO_h^{\mathrm{int}})}^2 + Ch^2\sum_{K\subset\cO_h^{\mathrm{bdry}}} \|v_\alpha\|_{H^2(K)}^2\\
		&\leq Ch^{1 - \f2q} \|\nabla v_\alpha\|_{H^1(\cO)}^2 + Ch^2 \|v_\alpha\|_{H^2(\cO_h^{\mathrm{bdry}})}^2\\
		&\leq Ch^{1 - \f2q} \|v_\alpha\|_{H^2(\cO)}^2
	\end{align*}
	for all $2<q<\infty$, where \eqref{eq:grad_sobolev_estimate} and the Sobolev embedding $H^1(\cO)\hookrightarrow L^q(\cO)$ were used in the fourth line.
Hence the gradient estimate becomes
	\begin{align}\label{eq:gradient_estimate}
		\|\nabla v_\alpha-\nabla u^h\|_{L^2(\cO\setminus\cO_h^{\mathrm{int}})} \leq C h^{\f12 - \f1q} \|v_\alpha\|_{H^2(\cO)}
	\end{align}
	Combining estimates \eqref{eq:L2_estimate} and \eqref{eq:gradient_estimate} completes the proof.
\end{proof}
Lemmas \ref{lemma:Oh^int} and \ref{lemma:Oh^bdry} (choosing $q=6$, for definiteness), together with \eqref{eq:Cea-est} and \eqref{eq:FEM_error}, finally yield the estimate
\begin{align*}
	\|v_\alpha-v^h_\alpha\|_{H^1(\cO)} &\leq C \inf_{w^h\in W^h} \|v_\alpha-w^h\|_{H^1(\cO)}\\
	&\leq C \|v_\alpha-u^h\|_{H^1(\cO)}\\
	&\leq C\Big[ \|v_\alpha-u^h\|_{H^1(\cO\setminus\cO_h^\text{int})} + \|v_\alpha-u^h\|_{H^1(\cO_h^\text{int})} \Big]\\
	&\leq C\Big[ h^{\f13} \|v_\alpha\|_{H^2(\cO)} + h \|v_\alpha\|_{H^2(\cO_h^{\mathrm{int}})} \Big]\\
	&\leq Ch^{\f13}\|v_\alpha\|_{H^2(\cO)}
\end{align*}
(for $h$ small enough). Finally, by elliptic regularity, we obtain the estimate
\begin{align}\label{eq:final_FEM_estimate}
	\|v_\alpha-v_\alpha^h\|_{H^1(\cO)} &\leq Ch^{\f13}\|f_\alpha\|_{L^2(\cO)}
\end{align}
where we recall that $f_\alpha$ was defined in \eqref{eq:f-alpha}.
\begin{remark}[Numerical integration]
	As we noted in Remark \ref{remark:numint1}, the approximate solution $v_\alpha^h$ is not computable from a finite subset of \eqref{eq:evaluation_set} in finitely many steps. However, the classical theory of finite elements shows that numerical integration does not change the error estimate \eqref{eq:final_FEM_estimate} (see e.g. \cite[Sec. 4.1]{Ciarlet}).
\end{remark}
This concludes the proof of Proposition \ref{prop:FEM_estimate}.
\subsection{Matrix elements, revisited}
Going back to \eqref{eq:matrix_elements} (and taking into account \eqref{eq:explicit_extension}), in order to compute the matrix elements $\mathcal K_{\alpha\beta}$, we have to approximate the integrals
\begin{align}
	I_1 &:= \int_{\cO} \overline {E_\beta}\,v_\alpha\,dx,   \notag\\
	I_2 &:= \int_{\cO} \nabla \overline {E_\beta}\cdot\nabla v_\alpha\,dx, \notag \\
	I_3 &:= \int_{\cO} \overline {E_\beta}f_\alpha\,dx. \label{eq:Bessel-Term}
\end{align}
This will be a simple consequence of Prop. \ref{prop:FEM_estimate} and Th. \ref{th:interpolation_estimate}. Let $\Pi^h$ denote the linear interpolation operator w.r.t. the triangulation $\mathcal T_h$, i.e. $\Pi^hu(j)=u(j)$ for all $j\in L_h$ and $\Pi^hu|_K$ is an affine function for all elements $K$.
\begin{prop}\label{prop:FEM_numerical_integrals}
	We have
	\begin{align}
		\left|I_1 - \int_{\cO} \Pi^h\overline {E_\beta} v_\alpha^h\,dx\right| &\leq C  h^{\f13}  \beta^2\|f_\alpha\|_{L^2(\cO)}, \label{eq:FEM-Term1_error}\\
		\left|I_2 - \int_{\cO} \Pi^h\nabla\overline {E_\beta}\cdot\nabla v_\alpha^h\,dx\right| &\leq C  h^{\f13}  \beta^2\|f_\alpha\|_{L^2(\cO)},\label{eq:FEM-Term2_error}\\
		\left|I_3 - \int_{\cO}\Pi^h\overline {E_\beta}\,\Pi^h f_\alpha\,dx\right| &\leq Ch^{2}\beta^2\|f_\alpha\|_{H^2(\cO)}. \label{eq:FEM-Term3_error}
	\end{align}
\end{prop}
\begin{proof}
	We begin with \eqref{eq:FEM-Term1_error} and assume w.l.o.g. $h<1$. Constants independent of $h,\alpha,\beta$ will be denoted $C$ and their value may change from line to line. By the triangle and H\"older inequalities, we have
	\begin{align*}
		\left|I_1 - \int_{\cO} \Pi^h\overline {E_\beta} v_\alpha^h\,dx\right| &\leq \| E_\beta\|_{L^2}\| v_\alpha - v_\alpha^h\|_{L^2} + \|v_\alpha^h\|_{L^2}\|E_\beta-\Pi^hE_\beta\|_{L^2}
		\\
		&\leq C  h^{\f13}  \| E_\beta\|_{L^2}\|f_\alpha\|_{L^2} + Ch^2\|E_\beta\|_{H^2}\big(\|v_\alpha\|_{L^2}+\|v_\alpha^h-v_\alpha\|_{L^2} \big)
		\\
		&\leq C  h^{\f13}  \| E_\beta\|_{L^2}\|f_\alpha\|_{L^2} + Ch^2\|E_\beta\|_{H^2}\big(\|f_\alpha\|_{L^2} +   h^{\f13}  \|f_\alpha\|_{L^2} \big)
		\\
		&\leq C  h^{\f13}  \| E_\beta\|_{H^2}\|f_\alpha\|_{L^2}
		\\
		&\leq C  h^{\f13}  \beta^2\|f_\alpha\|_{L^2},
	\end{align*}
	where we have used Prop. \ref{prop:FEM_estimate} \& Th. \ref{th:interpolation_estimate} in the second line, and Prop. \ref{prop:FEM_estimate} \& elliptic regularity in the third line. The estimate $\| E_\beta\|_{H^2}\leq C\beta^2$ used in the last line is evident from the definition of $E_\beta$.
	We omit the calculations leading to \eqref{eq:FEM-Term2_error} and \eqref{eq:FEM-Term3_error}, since they are entirely analogous.
\end{proof}
\begin{remark}
	We note that (by a straightforward calculation in polar coordinates) the integral \eqref{eq:Bessel-Term} vanishes unless $\alpha=\beta$. In fact, one has
	\begin{align*}
		\int_{\cO} \overline {E_\beta}f_\alpha\,dx = \f1{RJ_\alpha(kR)} \delta_{\alpha\beta}\int_{R-1}^R \rho(r) \big[2\rho'(r)(krJ_{\alpha-1}(kr) - \alpha J_{\alpha}(kr)) + r\Delta\rho(r) J_\alpha(kr)\big]\,dr,
	\end{align*}
	where $\delta_{\alpha\beta}$ denotes the Kronecker symbol. This one-dimensional integral can easily be approximated by a Riemann sum. While this fact is irrelevant to our theoretical treatment, it may improve performance and precision of numerical implementations.
\end{remark}
Let us now go back to \eqref{eq:matrix_elements}. Prop \ref{prop:FEM_numerical_integrals} allows us to define an approximate $n\times n$ matrix $\mathcal K_h$ by
\begin{align}\label{eq:K_h_def}
	(\mathcal K_h)_{\alpha\beta} := \int_{\cO} (\Pi^h\nabla \overline {E_\beta})\cdot\nabla v_\alpha^h \,dx  - k^2\int_{\cO} (\Pi^h\overline {E_\beta})v_\alpha^h \,dx - \int_{\cO} (\Pi^h\overline {E_\beta})(\Pi^hf_\alpha)\,dx,
\end{align}
whose matrix elements satisfy the error estimate
\begin{align}\label{eq:matrix_el_est1}
	|\mathcal K_{\alpha\beta} - (\mathcal K_h)_{\alpha\beta}| &\leq C(k) \beta^2\Big(  h^{\f13}  \|f_\alpha\|_{L^2(\cO)} + h^2\| f_\alpha\|_{H^2(\cO)}\Big).
\end{align}
The norms on the right hand side can be estimated by explicit calculation. One has
\begin{align*}
	\|f_\alpha\|_{H^1(\cO)} &\leq C|\alpha| 
	\\
	\|f_\alpha\|_{H^2(\cO)}  &\leq C|\alpha|^3
\end{align*}
Plugging these bounds into \eqref{eq:matrix_el_est1}, we finally arrive at
\begin{align}\label{eq:matrix_el_est2}
	|\mathcal K_{\alpha\beta} - (\mathcal K_h)_{\alpha\beta}| &\leq C(k) \Big(  h^{\f13}  \beta^2|\alpha| + h^2\beta^2|\alpha|^3 \Big).
\end{align}
We summarize these results in the following
\begin{theorem}\label{th:K-K_h_norm_estimate}
	For any $n\in\N$ one has the operator norm error estimate\footnote{
		We are abusing notation here, by identifying $\mathcal K_h$ with $\mathcal K_h\oplus 0_{\operatorname{span}\{e_1,\dots,e_n\}^\perp}$
	}
	\begin{align*}
		\|P_n\mathcal KP_n-\mathcal K_h\|_{L(\h)} &\leq C(k)(  h^{\f13}  n^3+h^2n^5),
	\end{align*}
	where $C(k)$ is independent of $n,h$ and bounded for $k$ in a compact set disjoint from the real line.
\end{theorem}
\begin{proof}
	This follows from the generalized Young inequality:
	\begin{align*}
		\|(\mathcal K-\mathcal K_h)u\|_{L^2} &\leq \max\Bigg\{ \sup_{|\alpha|\leq n}\,\sum_{\beta=-n}^n |\mathcal K_{\alpha\beta}-(\mathcal K_h)_{\alpha\beta}| \,,\,\sup_{|\beta|\leq n}\,\sum_{\alpha=-n}^n |\mathcal K_{\alpha\beta}-(\mathcal K_h)_{\alpha\beta}| \Bigg\}\cdot \|u\|_{L^2}.
	\end{align*}
	Since $|\mathcal K_{\alpha\beta}-(\mathcal K_h)_{\alpha\beta}|\leq C(  h^{\f13}  n^3+h^2n^5)$, by  \eqref{eq:matrix_el_est2}, this immediately implies the assertion.
\end{proof}
\begin{remark}
	We emphasize that all terms on the right hand side of \eqref{eq:K_h_def} can be computed from the data provided in \eqref{eq:evaluation_set} in finitely many steps. Indeed, the mass and stiffness matrices $\mathfrak m,\,\mathfrak s$ of the P1 finite element method can be computed exactly and satisfy
	\begin{align}
		\int_{\cO} \phi\psi\,dx &= \sum_{i,j\in L_h}\phi(i)\mathfrak m_{ij}\psi(j)
		\label{eq:mass_matrix}
		\\
		\int_{\cO} \nabla\phi\cdot\nabla\psi\,dx &= \sum_{i,j\in L_h}\phi(i)\mathfrak s_{ij}\psi(j)
		\label{eq:stiffness_matrix}
	\end{align}
	for all $\phi,\psi\in W^h$.
\end{remark}
\section{The algorithm}\label{sec:Algorithm}
\subsection{Proof of Theorem \ref{th:fixed_R}}
From Theorem \ref{th:K-K_h_norm_estimate} (by choosing  $h\sim n^{-11}$) and Lemma \ref{lemma:abstract_projection_estimate}, we obtain a computable approximation $\mathcal K_{h(n)}$ such that
\begin{align}\label{eq:h=n-9}
		\left\| \mathcal N^{-\f12}(\BH+\BJ+\mathcal K)\mathcal N^{-\f12} - P_n\mathcal N^{-\f12}(\BH+\BJ+\mathcal K_{h(n)})\mathcal N^{-\f12}P_n \right\|_{C_p}\leq Cn^{-\f12+\f1p}.
\end{align}
Using the Lipschitz continuity of perturbation determinants (cf. \cite[Th. 6.5]{Simon}) we conclude 
\begin{lemma}\label{lemma:determinant_rate}
Let $k\in\C^-$ and define $A_n^R(k):=\f R2 P_n\mathcal N^{-\f12}\big(\BH(k)+\BJ(k)+\mathcal K_{h(n)}(k)\big)\mathcal N^{-\f12}P_n-P_0$. Then there exists $C>0$, which is independent of $k$ for $k$ in a compact subset of $\C^-$, such that
	\begin{align*}
	\left| \mathrm{det}_{\lceil p\rceil}\left(I+\f R2 \mathcal N^{-\f12}\big(\BH(k)+\BJ(k)+\mathcal K(k)\big)\mathcal N^{-\f12}-P_0\right) - \mathrm{det}_{\lceil p\rceil}\left(I+A_n^R(k)\right) \right|\leq Cn^{-\f12+\f1{\lceil p\rceil}}.
\end{align*}
\end{lemma}
Fix a non-empty compact set  $Q$ in $\C^-$ and define the grid $G_n=\f1n(\Z+\i\Z)$. Because $\det_{\lceil p\rceil}(\cdots)$ is analytic in $k$, it can only have finitely many zeros in $Q$ and all are of finite order. That is, near a zero $k_0$ one has
\begin{align}\label{eq:analytic_decay}
	\left|\mathrm{det}_{\lceil p\rceil}\left(I+\f R2\mathcal N^{-\f12}\big(\BH(k)+\BJ(k)+\mathcal K(k)\big)\mathcal N^{-\f12}-P_0\right)\right| \leq \tilde C|k-k_0|^\nu
\end{align}
for some $\tilde C,\nu>0$ and for all $k$ in a sufficiently small neighborhood of $k_0$. 
Next, we define the $R$ and $Q$ dependent algorithm by
\begin{align}
	\Gamma_n^{Q,R} &: \Omega_R \to \operatorname{cl}(\C)\nonumber
	\\
	\Gamma_n^{Q,R}(U) &:=\left\{ k\in G_n\cap Q\,\bigg|\,\left|\det\nolimits_{\lceil p\rceil}\big(I+A_n^R(k)\big)\right|\leq \f{1}{\log(n)} \right\}.
	\label{eq:GammaQR}
\end{align}
\begin{theorem}\label{th:local_res_conv}
	For any $U\in\Omega_R$, we have that $\Gamma_n^{Q,R}(U)\to\Res(U)\cap Q$ in Hausdorff distance as $n\to+\infty$.
\end{theorem}
\begin{remark}[The Hausdorff distance for empty sets]
We note that since $Q$ is arbitrary, it may very well be the case that $\Res(U)\cap Q=\emptyset$, and then one needs to be cautious when using the Hausdorff distance. In this case we employ the following conventions for the Hausdorff distance:
 	\begin{align*}
	&d_{\mathrm{H}}(A,\emptyset)=+\infty,\quad\text{if }A\neq\emptyset,\\
	&d_{\mathrm{H}}(\emptyset,\emptyset)=0.
	\end{align*}
 \end{remark}
\begin{proof}[Proof of Theorem \ref{th:local_res_conv}]
First we treat the case $\Res(U)\cap Q\neq\emptyset$. Let us write $A^R(k):=\f R2 \mathcal N^{-\f12}\big(\BH(k)+\BJ(k)+\mathcal K(k)\big)\mathcal N^{-\f12}-P_0$ for notational convenience.
The proof consists of two steps. First we prove that any convergent sequence $k_n\in\Gamma_n^{Q,R}(U)$ necessarily converges to a zero of $\det_{\lceil p\rceil}(I+A^R)$, and second we prove that for every zero $k_0$ of $\det_{\lceil p\rceil}(I+A^R)$ there exists a sequence $k_n\in\Gamma_n^{Q,R}(U)$ converging to $k_0$. Together, these two facts imply Hausdorff convergence. 

Let us begin with step 1. Assume that $k_n\in\Gamma_n^{Q,R}(U)$  converges to some $\tilde k\in\C$. By the definition of $\Gamma_n^{Q,R}$, we have $|\det(I+A_n^R(k_n))| \leq \f{1}{\log(n)}$. Hence, by Lemma \ref{lemma:determinant_rate} we have that
\begin{align*}
	|\det\nolimits_{\lceil p\rceil}(I+A^R(k_n))| &\leq \f{1}{\log(n)} + Cn^{-\f12+\f1{\lceil p\rceil}}
	\to 0
\end{align*}
as $n\to+\infty$. Hence we have
\begin{align*}
	\det\nolimits_{\lceil p\rceil}(I+A^R(\tilde k)) = \lim_{n\to+\infty}\det(I+A^R(k_n)) = 0.
\end{align*}
Conversely, suppose that $\det(I+A^R(k_0))=0$ for some $k_0\in Q$. Then there exists a sequence $k_n\in G_n\cap Q$ such that $|k_0-k_n|\leq \f1n$. Hence, by \eqref{eq:analytic_decay} we have
\begin{align*}
	|\det\nolimits_{\lceil p\rceil}(I+A^R(k_n))|\leq \tilde C\f{1}{n^\nu}
\end{align*}
and employing Lemma \ref{lemma:determinant_rate} again, we obtain
\begin{align}\label{eq:1/n+Cn}
	|\det\nolimits_{\lceil p\rceil}(I+A_n^R(k_n))|\leq \tilde C\f{1}{n^\nu} + Cn^{-\f12+\f1{\lceil p\rceil}}.
\end{align}
Clearly, for large enough $n$, the right hand side of \eqref{eq:1/n+Cn} will be less than $\f1{\log(n)}$, hence $k_n\in\Gamma_n^{Q,R}(U)$ for $n$ large enough. This completes the proof for the case $\Res(U)\cap Q\neq\emptyset$.

If $\Res(U)\cap Q=\emptyset$ then $\det(I+A^R(k_0))$ is bounded uniformly away from $0$ in $Q$. Hence for all $n$ large enough, the condition  \eqref{eq:GammaQR} for points to belong to $\Gamma_n^{Q,R}(U)$ will not be fulfilled and therefore $\Gamma_n^{Q,R}(U)=\emptyset$. Our convention regarding the Hausdorff distance for empty sets then leads to $d_{\mathrm{H}}(\Res(U)\cap Q,\Gamma_n^{Q,R}(U))=0$ for all $n$ large enough.
\end{proof}
It remains to extend our argument from a single compact set $Q$ to the entire complex plane. This is done via a diagonal-type argument. We choose a tiling of $\C^-$, where we start with a rectangle $Q_1=\big\{z\in\C^-\,\big|\,|\re(z)|\leq\f12,\,|\im(z)+\f32|\leq\f12 \big\}$ and then add rectangles in a counterclockwise spiral manner as shown in Figure \ref{fig:tiling}. Note again that each individual rectangle $Q_i$ is well separated from $\{k\,|\,k^2\in\sigma(H_{\mathrm{D}})\}\subset\R$.
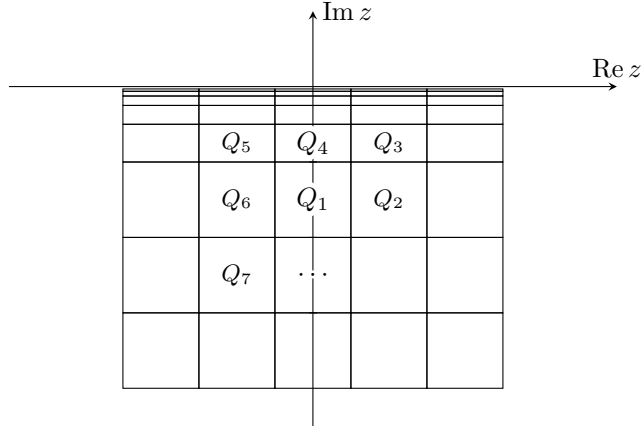
\begin{figure}[htbp]
	\centering
	\contourlength{2pt}
	\begin{tikzpicture}[>=stealth]
		\draw[-stealth] (-4,0)--(4,0) node[above]{$\mathrm{Re}\, z$};
		\draw[-stealth] (0,-4.5)--(0,1) node[right]{$\mathrm{Im}\, z$};
		\draw   (-0.05,-1)--(0.05,-1);
		\draw   (-0.05,-2)--(0.05,-2);
		\draw   (-0.05,-3)--(0.05,-3);
		\foreach \y in {0,1,2}
		\foreach \x in {-2,-1,0,1,2} {
			\draw (-0.5+\x,-2-\y) rectangle (0.5+\x,-1-\y);
		}
		\foreach \x in {-2,-1,0,1,2} {
			\draw (-0.5+\x,-1) rectangle (0.5+\x,-0.5);
			\draw (-0.5+\x,-0.5) rectangle (0.5+\x,-0.25);
			\draw (-0.5+\x,-0.25) rectangle (0.5+\x,-0.125);
			\draw (-0.5+\x,-0.125) rectangle (0.5+\x,-0.0625);
			\draw (-0.5+\x,-0.0625) rectangle (0.5+\x,-0.03125);
		}
		\node at (0,-1.5) {\contour{white}{$Q_1$}};
		\node at (1,-1.5) {\small$Q_2$};
		\node at (1,-0.75) {\small$Q_3$};
		\node at (0,-0.75) {\contour{white}{$Q_4$}};
		\node at (-1,-0.75) {\small$Q_5$};
		\node at (-1,-1.5) {\small$Q_6$};
		\node at (-1,-2.5) {\small$Q_7$};
		\contourlength{1pt}
		\node at (0.03,-2.5) {\contour{white}{$\cdots$}};
	\end{tikzpicture}
	\caption{Tiling of the lower half plane}
	\label{fig:tiling}
\end{figure}
Next, we define our algorithm as follows. We let
\begin{align*}
	\Gamma_1^R(U) &:= \Gamma_1^{Q_1,R}(U) \\
	\Gamma_2^R(U) &:= \Gamma_2^{Q_1,R}(U) \cup \Gamma_2^{Q_2,R}(U) \\
	\Gamma_3^R(U) &:= \Gamma_3^{Q_1,R}(U) \cup \Gamma_3^{Q_2,R}(U) \cup \Gamma_3^{Q_3,R}(U) \\
	&\vdots\\
	\Gamma_n^R(U) &:= \bigcup_{j=1}^n \Gamma_n^{Q_j,R}(U).
\end{align*}
Theorem \ref{th:local_res_conv} ensures that $\Gamma_n^R(U)\cap B\to \Res(U)\cap B$ as $n\to+\infty$ in Hausdorff sense for each compact subset $B$ of $\C^-$. The proof of Theorem \ref{th:fixed_R} now follows from the characterization of the Attouch-Wets metric appearing in \eqref{eq:Attouch-Wets}.

\subsection{Proof of Theorem \ref{th:mainth}}
It remains to extend the above algorithm from $\Omega_R$ to $\Omega$.
In this section we will define a new algorithm, $\Gamma_n$, based on the algorithm $\Gamma_n^R$ constructed above. The main difficulty consists in ``locating $U$'' by testing only finitely many points.

The algorithm below uses two radii, $R$ and $r$. $R$ is used for running $\Gamma_n^R$, while $r$ is successively increased to search for hidden components of $U$. The following pseudocode gives the definition of $\Gamma_n$.

\medskip
\begin{algorithm}[H]
\SetAlgoLined
 Let $U\in\Omega$ and initialise $R=r=1$\;
 \For{$n\in\N$}{
 	Consider the lattice $L_n = n^{-11}\Z^2\cap B_r$ (recall the choice $h\sim n^{-11}$ before eq. \eqref{eq:h=n-9}). For every $j\in L_n$, test whether $j\in U$. For all $j\in L_n \cap U$, compute $|j|$\;
	\eIf{$|j|\leq R-1$ for all $j\in L_n \cap U$,}{
		 define $\Gamma_n(U):=\Gamma_n^R(U\cap B_{R-1})$, increment $r$ by 1 and proceed to $n+1$\;
	}
	{
		increment $r$ by 1, set $R:=r$ and repeat the current step\;
	}
 }
 \caption{Remove $R$ dependence}\label{Alg:0}
\end{algorithm}

\medskip
This process generates increasing sequences, that we denote by $R_n$ and $r_n$, and defines an algorithm $\Gamma_n:\Omega\to\operatorname{cl}(\C)$. Note that $r_n$ diverges to $+\infty$, because it gets incremented by at least 1 in every step. 
\begin{lemma}\label{lemma:R_n_eventually_constant}
	The sequence $\{R_n\}_{n\in\N}$ is eventually constant and if $N>0$ is such that $R_n=R_N$ for all $n>N$, one has $U\subset B_{R_N-1}$
\end{lemma}
\begin{proof}
	The fact that $\{R_n\}_{n\in\N}$ is eventually constant follows immediately from the boundedness of $U$ and the above pseudocode. Now let $N\in\N$ be as in the assertion. Assume for contradiction that $U\nsubseteq B_{R_N-1}$. Then, there exists $x\in U$ with $|x|>R_N-1$, and, since $U\setminus \overline{B}_{R_N-1}$ is open, there exists $\eps>0$ such that $B_\eps(x)\subset U\setminus \overline{B}_{R_N-1}$. But then, as soon as $n>N$ is large enough such that $n^{-11}<\eps$ and $r_n>|x|$, there would exist $j\in L_n$ with $j\in B_\eps(x)$. According to the pseudocode, then, $R_n$ would be incremented by 1, contradicting the fact that $R_n=R_N$ for all $n>N$.
\end{proof}
By definition, the output of the algorithm $\Gamma_n$ is $\Gamma_n^{R_n}(U\cap B_{R_n-1})$. Since $R_n$ is eventually constant, there is $N\in\N$ such that $R_n\equiv R_N$ for all $n\geq N$. Hence we have
\begin{align*}
	\lim_{n\to+\infty} \Gamma_n(U) &=
	\lim_{n\to+\infty} \Gamma_n^{R_n}(U\cap B_{R_n-1}) \\
	&= \lim_{n\to+\infty} \Gamma_n^{R_N}(U\cap B_{R_N-1}) \\
	&= \lim_{n\to+\infty} \Gamma_n^{R_N}(U) \\
	&= \mathrm{Res}(U),
\end{align*}
where the second and third lines follow from Lemma \ref{lemma:R_n_eventually_constant} and the last line follows from the proof of Theorem \ref{th:fixed_R}. This completes the proof of Theorem \ref{th:mainth}.

\section{Numerical Results}\label{sec:num}
Although the algorithm described in Sections \ref{sec:Matrix_Elements} and \ref{sec:Algorithm} was never designed for computational efficiency, we  show in this section that it is possible to obtain surprisingly good numerical results with a MATLAB implementation which differs only in two minor respects:
\begin{enuma}
	\item The decomposition $M_\mathrm{inner} = M_{\mathrm{inner},0}+\mathcal K$ obtained in Lemma \ref{lemma:Lambda_inner} introduces artificial singularities at the spectrum of $H_\mathrm{D}$ (the Dirichlet Laplacian introduced in eq. \eqref{eq:H_D_def}), which lead to numerical instabilities. For this reason, rather than approximating the solution $v_\alpha$ of \eqref{eq:v_m-problem}, we directly implemented a finite element approximation of the solution $u$ of \eqref{eq:Lambda_inner_Def}. The matrix elements of $M_\mathrm{inner}$ can then be calculated via Green's formula, similarly to \eqref{eq:matrix_elements}.
	\item For reasons of performance, instead of using the triangulation procedure outlined in Section \ref{sec:FEM}, we use the meshing tool \emph{Distmesh}, cf. \cite{Persson}, to triangulate our domain.
\end{enuma}

\subsection{A circular resonator chamber}
The first geometry we consider is that of a circular resonator chamber, connected to open space by a narrow channel. Figure \ref{fig:circular_mesh} shows the triangulation of the domain $\mathcal{O}=B_R\setminus\overline U$. The specific parameter values in this example are: outer radius of the DtN sphere $R=3$, outer radius of the resonator chamber $r_{\text{outer}}=2$, inner radius of the resonator chamber $r_{\text{inner}}=1.8$, width of the opening $d=1.3$, meshing parameter $h=0.1$.
\begin{figure}[H]
	\centering
	\includegraphics{ 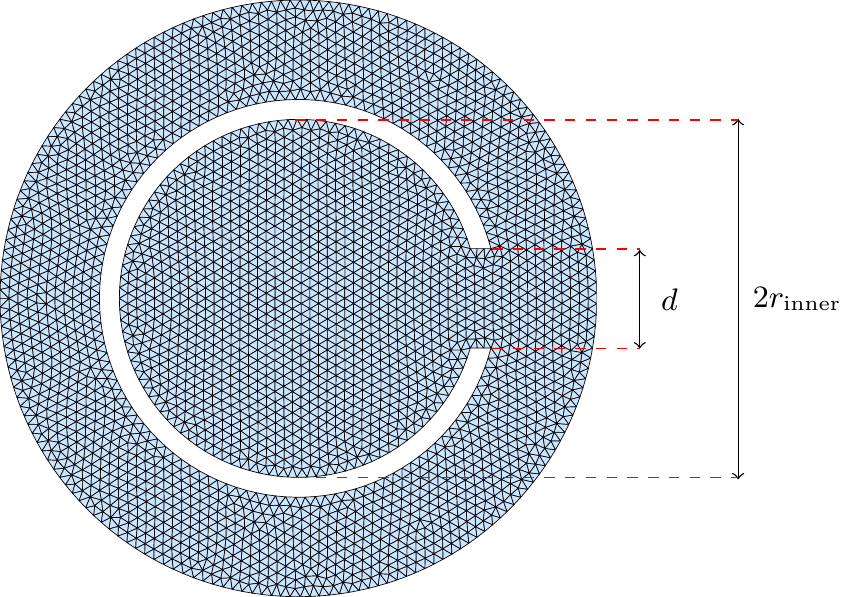}
	\caption{Example triangulation of a circular Helmholtz resonator used in our implementation.}
	\label{fig:circular_mesh}
\end{figure}
Figure \ref{fig:contour_plot} shows a logarithmic contour plot of the computed determinant $|\det(I+A_n(k))|$ in the complex plane for $n=20$. For the sake of comparison we included the Dirichlet eigenvalues of a \emph{closed} resonator chamber in the image (red dots). One can see that next to each of the Dirichlet eigenvalues there is zero of the determinant, as we would expect. 
\begin{figure}[H]
	\centering
	\includegraphics{ 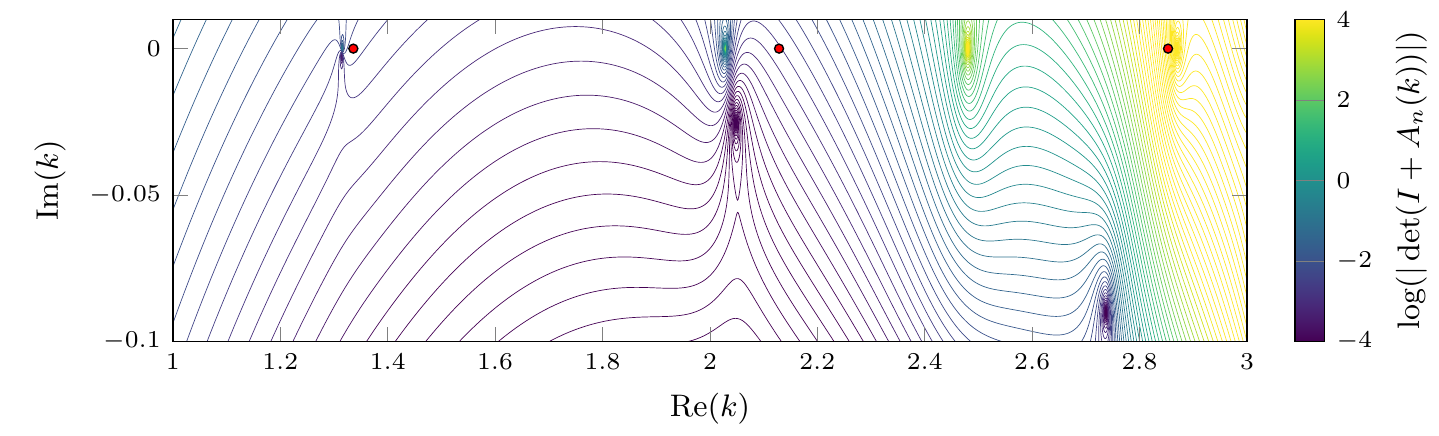}
	\caption{Logarithmic contour plot of $|\det(I+A_n(k))|$ for $n=20$ and opening width $d=1.3$ in a strip below the real axis. Red dots: Dirichlet eigenvalues of the closed resonator chamber.} 
	\label{fig:contour_plot}
\end{figure}
The resonance effects can in fact be seen on the level of the associated PDE. Figure \ref{fig:FEM_fig} shows the finite element approximation of the solution $u$ of \eqref{eq:Lambda_inner_Def} with right hand side $\phi=e_\alpha$, once for $k$ far away from a resonance, and once for $k$ near a resonance. As can be seen, when $k$ is near the resonance the wave penetrates into the chamber, whereas when $k$ is far from a resonance the solution appears to be nearly $0$ inside.
\begin{figure}[H]
	\centering
	\includegraphics{ 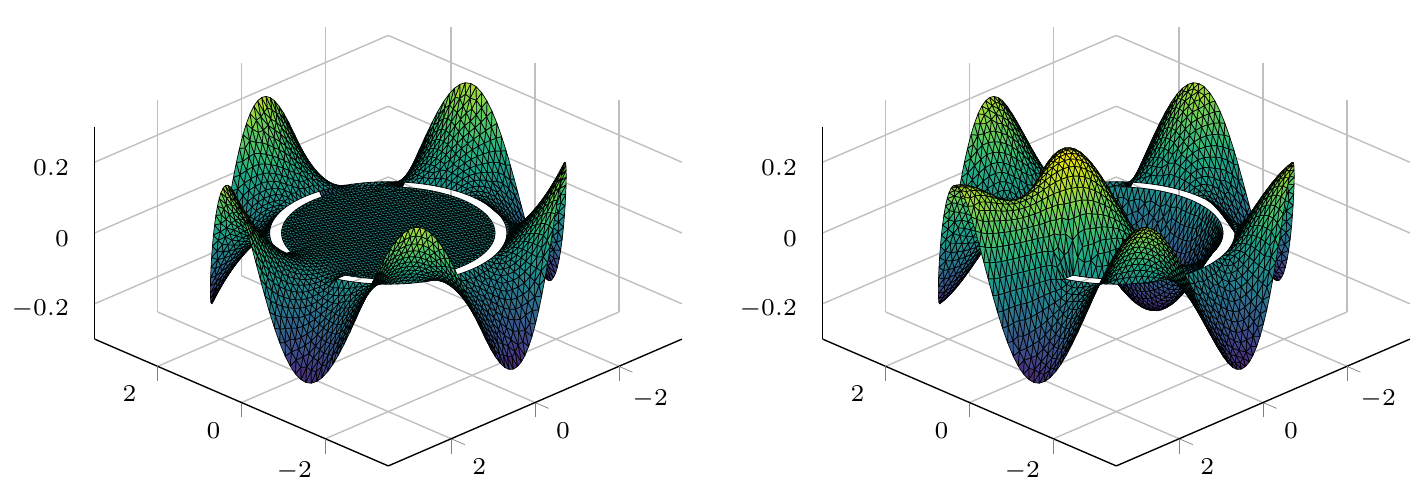}
	\caption{Solution of \eqref{eq:Lambda_inner_Def} with boundary values $\phi=e_\alpha$ for $\alpha=5$. Left: $k=1.0$ (far from resonance),  right: $k=2.049-0.026\i$ (near second resonance)}
	\label{fig:FEM_fig}
\end{figure}
Finally, our results show what happens to the locations of the resonances when the opening width $d$ is varied. Figure \ref{fig:heat_map} shows the contour plot of $|\det(I+A_n(k))|$ for an opening width of $d=1.0$. One can clearly see that the zeros of $\det(I+A_n(k))$ move closer to the real axis and become narrower, as is expected from the abstract theory (cf. \cite[Ch. 23.3]{HislopSigal}). 
\begin{figure}[H]
	\centering
	\includegraphics{ 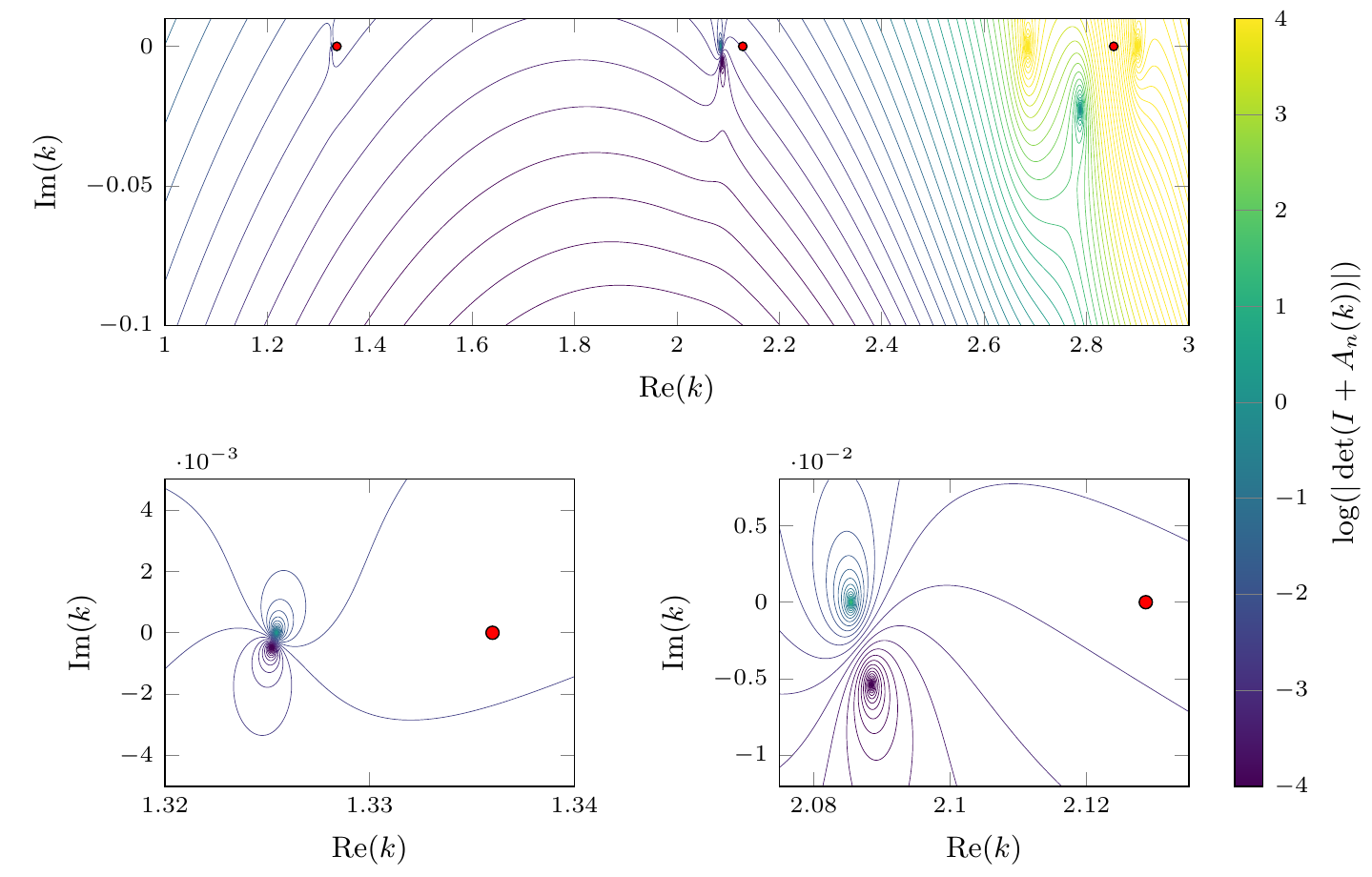}
	\caption{Top: Logarithmic contour plot of $|\det(I+A_n(k))|$ for $n=20$ and opening width $d=1.0$ in a strip below the real axis; Bottom left: zoomed in around first resonance; Bottom right: zoomed in around second resonance. Red dots: Dirichlet eigenvalues of the closed resonator chamber.} 
	\label{fig:heat_map}
\end{figure}
Table \ref{table:res} summarizes the values of the first three eigenvalues and associated resonances in both cases (that is, with openings $d=1.0$ and $d=1.3$). One can observe that when the opening is narrower, the resonances are closer to their eigenvalue counterparts:
\begin{table}[H]
	\begin{tabular}{l c c c}\toprule
		& \textbf{Eigenvalue}\; & \multicolumn{2} {c}{\textbf{Resonance}}\\
		&&$\boldsymbol{d\!=\!1.3}$&$\boldsymbol{d\!=\!1.0}$\\
				\hline
		first  & $1.3360$ & $1.315-0.002\i$ & $1.325-0.001\i$ \\
		second & $2.1287$ & $2.049-0.026\i$ & $2.089-0.005\i$ \\
		third  & $2.8531$ & $2.738-0.090\i$ & $2.788-0.023\i$  \\
		\bottomrule
	\end{tabular}
	\caption{Approximate numerical values of resonances and eigenvalues corresponding to the contour plots in Figures \ref{fig:contour_plot} and \ref{fig:heat_map} respectively. The values in the second and third column were obtained as the local minima of $|\det(I+A_n(k))|$.}
\label{table:res}\end{table}

\subsubsection*{Convergence analysis}
In order to verify Theorem \ref{th:mainth} and estimate the practical rate of convergence, we compute a successive series of approximations for the first resonance for the opening $d=1.3$. Because for large values of $n$ the choices of the meshing parameter made above ($h\sim n^{-11}$) and the threshold $1/\log n$ are not feasible for implementation, we demonstrate that even a simpler version of \eqref{eq:GammaQR} exhibits good convergence properties. 

\begin{figure}[h]
	\centering
	\includegraphics{ 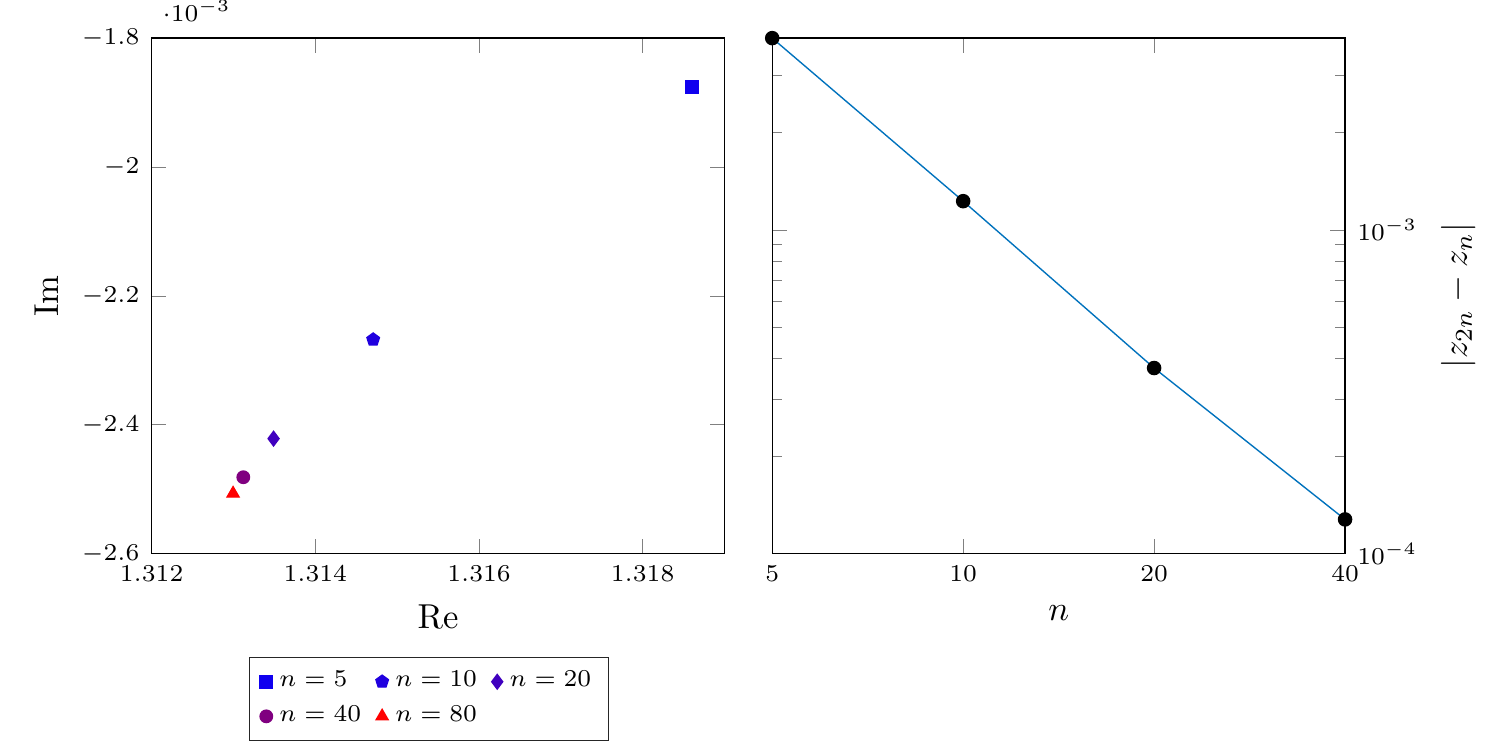}
	\caption{Left: Approximations $z_n$ of the first resonance for opening $d=1.3$ and $n\in\{5,10,20,40,80\}$; Right: Plot of the successive error $|z_{2n}-z_n|$ for $n\in\{5,10,20,40\}$.}
	\label{fig:convergence_plot}
\end{figure}
\begin{table}[h]
	\begin{tabular}{l c c c c c}\toprule
		$n$ & $z_n$ \\
		\hline
		5  & $1.3186 - 0.0019\i$  \\
		10 & $1.3147 - 0.0023\i$  \\
		20 & $1.3135 - 0.0024\i$  \\
		40 & $1.3131 - 0.0025\i$  \\
		80 & $1.3130 - 0.0025\i$  \\
		\bottomrule
	\end{tabular}
	\caption{Approximations $z_n$ of the first resonance for opening $d=1.3$}
	\label{table:convergence_table}
\end{table}

More precisely, we chose $h = n^{-1}$ and for $n\in\{5,10,20,40,80\}$ performed gradient descent steps until $|\det(I+A_n(z))|<10^{-7}$ (note that for all values of $n$ considered this is much less than $1/\log n$). The resulting points $z_n$, and their relative differences are shown in Figure \ref{fig:convergence_plot} and their first digits are given in Table \ref{table:convergence_table}. A comparison between the values in Tables \ref{table:res} and \ref{table:convergence_table} suggests that only the first two digits of the rough approximations in Table \ref{table:res} should be trusted (this is in accordance with the lattice distance in Figure \ref{fig:contour_plot}, which was chosen to be $10^{-3}$).
Finally, the decay of errors in the right-hand plot of Figure \ref{table:convergence_table} suggests a convergence rate of approximately $n^{-1.7}$.

\subsection{Four circular Neumann holes}\label{sec:neumann}
In many applications it is necessary to consider not only Dirichlet, but also Neumann boundary conditions. An important real life situation in which this is the case is presented by  resonance effects of water waves caused by submerged objects. A concrete question which has entailed a vast body of scientific literature is this: do the pillars of offshore structures, such as oil drilling platforms, introduce a positive interference of water waves that could damage the structure itself? At what frequencies are such positive interferences expected to occur? (cf. \cite{SM75,LE90,EP96}.) This section contains a short study of the latter question, together with some numerical experiments.
\begin{figure}[H]
	\centering
	\includegraphics{ 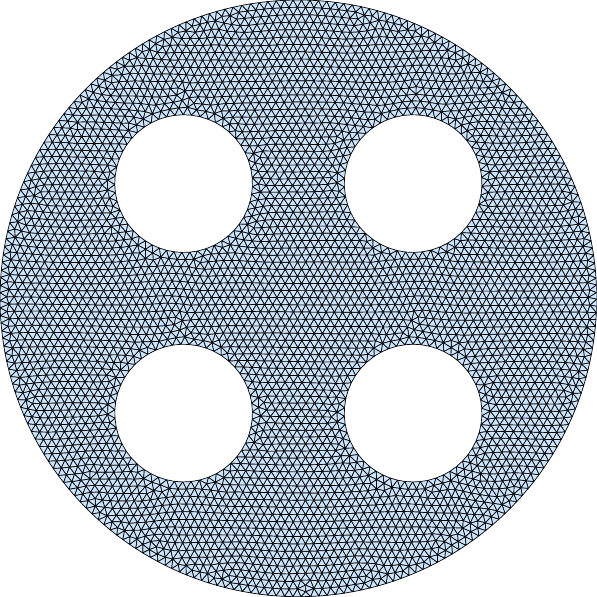}
	\caption{Triangulation of $B_R\setminus\overline U$ for $U$ consisting of four circular Neumann holes.}
	\label{fig:EP_mesh}
\end{figure}

The approach outlined in Sections \ref{sec:DtN}, \ref{sec:Matrix_Elements} and \ref{sec:Algorithm} can be adapted to Neumann boundary conditions on the boundary of the obstacle $U$, given a set of stronger assumptions. In this section, we will briefly outline the ideas and present some numerical results. 

Let now $\Omega_{R}$ denote the class of all bounded domains $U\Subset B_{R-1}\subset\R^2$  with a $\mathcal{C}^2$ boundary. We consider the associated Neumann Laplacian:
\begin{align*}
	H:=-\Delta_\mathrm{N}\quad\text{on}\quad L^2(\R^2\setminus \overline U).
\end{align*}
In addition, let us assume that each obstacle $U$ comes with an associated regular parametrization $\gamma_U:[a_1^U,b_1^U]\cup[a_2^U,b_2^U]\cdots\cup[a_{n_U}^U,b_{n_U}^U]\to\R^2$ (parametrized by its arc length) of its boundary. Define
\begin{align*}
	\Lambda_\mathrm{N} := \Lambda \cup \left\{U\mapsto\gamma_U(x)\,\middle|\,x\in\bigcup_{k=1}^{n_U}[a_{k}^U,b_{k}^U]\right\},
\end{align*}
where $\Lambda$ was defined in \eqref{eq:evaluation_set}. Under this additional hypothesis, Sections \ref{sec:DtN}, \ref{sec:Matrix_Elements} and \ref{sec:Algorithm} (with modifications) show that
\begin{align*}
	\SCI(\Omega_{R},\Lambda_{\mathrm{N}},\Res(\cdot),(\mathrm{cl}(\C),d_\mathrm{AW}))=1.
\end{align*}
Indeed, the discussions in Sections \ref{sec:DtN} and \ref{sec:Algorithm} remain true with trivial changes for Neumann conditions on $\del U$. The only major difference is in Section \ref{sec:FEM}, where the finite element approximation of $v_\alpha$ is constructed. A version of Prop. \ref{prop:FEM_estimate} in the Neumann case can be obtained from our new assumptions as follows. 
\begin{enumi}
	\item For each $U$, choose a uniform discretization of the intervals $[a_{k}^U,b_{k}^U]$. This induces an oriented discretization of $\del U$ via $\gamma_U$.
	\item Using well known meshing algorithms (e.g. \cite[Ch. 5]{B98}), construct a fitted mesh of $B_R\setminus \overline U$ based on the boundary discretization from (i).
	\item Apply \cite[Th. 4.1]{BE84} to obtain the desired FEM error estimate.
\end{enumi}
We implemented the algorithm thus obtained (with the caveats (a), (b) mentioned at the beginning of the section) for an obstacle consisting of four circular holes, as shown in Figure \ref{fig:EP_mesh}.

This geometry had previously been studied in the context of water waves interacting with a circular array of cylinders \cite{EP96}. Figure \ref{fig:EP_heat} shows the output of our algorithm for four circular holes of radius 0.6, situated at the corners of a square of edge length 2.
\begin{figure}[H]
	\centering
	\includegraphics{ 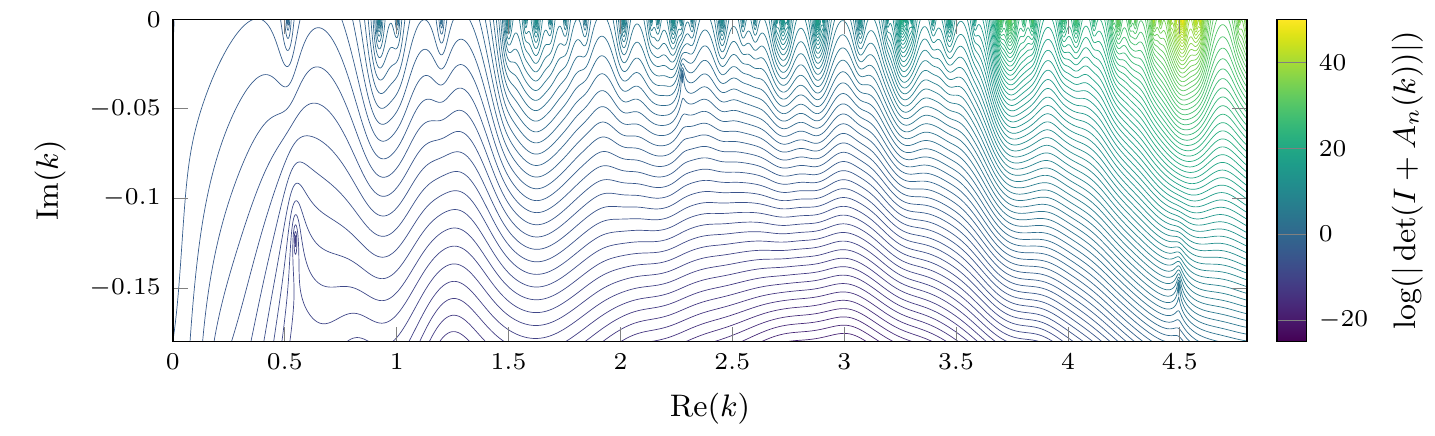}
	\caption{Logarithmic contour plot of $|\det(I+A_n(k))|$ for $n=30$ in a strip below the real axis}
	\label{fig:EP_heat}
\end{figure}
The heat map in Figure \ref{fig:EP_heat} suggests 3 resonances located approximately at the following values:
\begin{table}[H]
\begin{tabular}{l l}\toprule
	 & \textbf{Resonance}\\
	\hline
	first & $0.545-0.123\i$\\
	second & $2.277-0.032\i$\\
	third & $4.495-0.149\i$\\
	\bottomrule
\end{tabular}
\caption{Approximate numerical values of resonances corresponding to the contour plot in Figure \ref{fig:EP_heat}.}
\end{table}
A comparison with \cite[Fig. 8]{EP96} shows good agreement.

\appendix

\section{Details on implementation}
As indicated in Section \ref{sec:num}, our implementation differs from the theoretical algorithm \eqref{eq:GammaQR}. Here we give some of the technical details. First, we note that the matrix elements of $M_\mathrm{inner}$ can be computed using only integrals over $\cO$ and the solution $u_\alpha$ of \eqref{eq:Lambda_inner_Def} (with right hand side $e_\alpha$). Indeed, by Green's identity one has
\begin{align*}
	\langle e_\beta, M_\mathrm{inner} e_\alpha\rangle_{L^2(\del B_R)} &= \int_{\del B_R} \overline e_\beta \del_\nu u_\alpha \,dS
	\\
	&= \int_{\cO} \overline E_\beta\Delta u_\alpha \,dx + \int_{\cO} \overline{\nabla E_\beta}\cdot\nabla u_\alpha\,dx
	\\
	&= -k^2 \int_{\cO} \overline E_\beta u_\alpha \,dx + \int_{\cO} \overline{\nabla E_\beta}\cdot\nabla u_\alpha\,dx.
\end{align*}
A finite element approximation of the last line yields
\begin{align}\label{eq:FEM_M_inner}
	\langle e_\beta, M_\mathrm{inner} e_\alpha\rangle_{L^2(\del B_R)} &\approx 
	-k^2\sum_{i,j\in L_h} \overline E_\beta(i)\mathfrak m_{ij}u_\alpha^h(j) + \sum_{i,j\in L_h} \overline E_\beta(i)\mathfrak s_{ij}u_\alpha^h(j)
	\\
	&= \sum_{i,j\in L_h} \overline E_\beta(i)\big(\mathfrak s_{ij}-k^2\mathfrak m_{ij}\big)u_\alpha^h(j),
\end{align}
where $\mathfrak m, \mathfrak s$ denote the mass and stiffness matrices of the FEM scheme (cf. \eqref{eq:mass_matrix}, \eqref{eq:stiffness_matrix}) and $u_\alpha^h$ denotes the FEM approximation of $u_\alpha$. Then, using the notation from Section \ref{sec:approximation_procedure}, one has
\begin{align*}
	\Lin(k) + \Lout(k) &= \f1R(\mathcal N - P_0) + \mathcal H(k) + \Lin(k),
\end{align*}
which after some simplifications yields the formula
\begin{align}\label{eq:modified_numerics_formula}
	\f{R}{2}\mathcal N^{-\f12}(\Lin(k) + \Lout(k))\mathcal N^{-\f12} = \f12\left(I - P_0 + R\mathcal N^{-\f12}(\mathcal H(k) + \Lin(k))\mathcal N^{-\f12}\right),
\end{align}
whose matrix elements can be computed entirely from the values of the Hankel functions and \eqref{eq:FEM_M_inner}. Note that by the theory in Section \ref{sec:DtN} we have $\Lin = \f1R\mathcal N + \mathcal J + \mathcal K$, so the right hand side of eq. \eqref{eq:modified_numerics_formula} is in fact of the form $I+[\text{compact}]$, as needed.
Eq. \eqref{eq:modified_numerics_formula} suggests the following procedure.

\medskip
\begin{algorithm}[H]
\SetAlgoLined
 Fix number of lattice points $n\in\N$\;
 Fix cutoff threshold $\eps>0$\;
 Fix triangulation mesh size $h>0$\;
 Fix matrix size $N\in\N$\;
 Initialize lattice $G_n\subset\C^-$\;
 Generate triangular mesh $\cO^h$\;
 Initialize set of resonances $\Res:=\{\}$\;
 \For{$k\in G_n$}{
	 \For{$-N\leq\alpha,\beta\leq N$}{
		  Compute FEM approximation $u_\alpha^h$\;
		  Compute $(\alpha,\beta)$-matrix element of \eqref{eq:modified_numerics_formula}  using \eqref{eq:FEM_M_inner}\;
	}
	Compute $D:=\Big|\det\Big(\f12\left(I - P_0 + R\mathcal N^{-\f12}(\mathcal H(k) + \Lin(k))\mathcal N^{-\f12}\right)_{\scriptscriptstyle{(2N+1)\times (2N+1)}}\Big)\Big|$\;
	  \If{$D<\eps$}{
		   $\Res := \Res\cup \{k\}$\;
	  }
 }
 \Return{$\Res$}
 \caption{Compute Resonances}\label{Alg:1}
\end{algorithm}

\medskip
The Matlab implementation of Algorithm \ref{Alg:1} which yielded Figures \ref{fig:contour_plot}-\ref{fig:EP_heat} is available at \url{https://github.com/frank-roesler/SeashellComp}. The code for the finite element approximation was partially adapted from \cite{bartels}.

\section{The Solvability Complexity Index Hierarchy}
\label{app:sci-hierarchy}
\begin{de}[The Solvability Complexity Index Hierarchy]
\label{1st_SCI}
The $\SCI$ Hierarchy is a hierarchy $\{\Delta_k\}_{k\in{\N_0}}$ of classes of computational problems $(\Om,\Lambda,\Xi,\mathcal M)$, where each $\Delta_k$ is defined as the collection of all computational problems satisfying:
\begin{align*}
(\Om,\Lambda,\Xi,\mathcal M)\in\Delta_0\quad &\qquad\Longleftrightarrow\qquad \mathrm{SCI}(\Om,\Lambda,\Xi,\mathcal M)= 0,\\
(\Om,\Lambda,\Xi,\mathcal M)\in\Delta_{k+1} &\qquad\Longleftrightarrow\qquad \mathrm{SCI}(\Om,\Lambda,\Xi,\mathcal M)\leq k,\qquad k\in\N,
\end{align*}
with the special class $\Delta_1$  defined as the class of all computational problems in $\Delta_2$ with a convergence rate:
\begin{equation*}
(\Om,\Lambda,\Xi,\mathcal M)\in\Delta_{1} \qquad\Longleftrightarrow\qquad
 \exists  \{\Gamma_n\}_{n\in \mathbb{N}} \text{ s.t. } \forall  T\in\Omega, \ d(\Gamma_n(T),\Xi(T)) \leq 2^{-n}.
\end{equation*}
Hence we have that $\Delta_0\subset\Delta_1\subset\Delta_2\subset\cdots$
\end{de}

In some cases these classes can be further refined to pinpoint precisely how the convergence of the algorithms occurs. For instance, in the case where the metric space is the space of all closed subsets of $\C$ equipped with the Attouch-Wets metric, one could ask whether the approximations $\Gamma_{n_1,n_2,\dots,n_k}(T)$ lie in a small neighborhood of the problem function $\Xi(T)$ for every $T$, or conversely, whether  the problem function $\Xi(T)$ lies in a small neighborhood of the approximations $\Gamma_{n_1,n_2,\dots,n_k}(T)$  for every $T$. Letting $B_\epsilon(A)$ denote the usual Euclidean $\epsilon$-neighborhood of a set $A$ in $\C$, the most basic definition, defining the first level of this hierarchy, is:

\begin{de}[The $\SCI$ Hierarchy For $(\mathrm{cl}(\C),d_{\mathrm{AW}})$: $\Sigma_1$ and $\Pi_1$]
\label{def:pi-sigma}
Consider the setup in Definition \ref{1st_SCI} assuming further that $\mathcal{M}=(\mathrm{cl}(\C),d_{\mathrm{AW}})$. Then we can define the following subsets of $\Delta_2$:
	\begin{equation*}
	\begin{split}
	\Sigma_{1} &= \Big\{(\Om,\Lambda,\Xi,\mathcal M) \in \Delta_{2} \ \vert \  \exists\{ \Gamma_{n}\}_{n\in\N}\text{ s.t. }    \forall T \in \Omega, \\
	&\qquad\qquad\qquad\qquad \Gamma_n(T)\to\Xi(T)\text{ and }\Gamma_{n}(T)  \subset B_{2^{-n}}(\Xi(T)) \Big\}
	\\
	\Pi_{1} &= \Big\{(\Om,\Lambda,\Xi,\mathcal M) \in \Delta_{2} \ \vert \  \exists \{\Gamma_{n}\}_{n\in\N}  \text{ s.t. }  \forall T \in \Omega,
	\\
	&\qquad\qquad\qquad\qquad \Gamma_n(T)\to\Xi(T)\text{ and } \Xi(T)\subset B_{2^{-n}}(\Gamma_{n}(T))   \Big\}. 
	\end{split}
	\end{equation*}
It can be shown that $\Delta_1=\Sigma_1\cap\Pi_1$.

One can further define higher classes $\Sigma_k$ and $\Pi_k$ in an ``intertwining'' way, so that $\Sigma_k,\Pi_k\subset\Delta_{k+1}$ and for $k\in\{1,2,3\}$, $\Delta_k=\Sigma_k\cap\Pi_k$, see Figure \ref{fig:sci-hir}. We omit this here, and refer to \cite{AHS}.
\end{de}

\begin{figure}[H]
\centering
\begin{tikzpicture}
\draw (0,0) circle (1) (-0.3,.9)  node [text=black,above] {$\Sigma_k$}
      (1,0) circle (1) (1.3,.9)  node [text=black,above] {$\Pi_k$}
      (.5,0) circle (1.95) (.5,1.9) node [text=black,above] {$\Delta_{k+1}$}
       (.5,-0.2) node [text=black,above] {$\Delta_{k}$};
\end{tikzpicture}
\caption{The $\SCI$ Hierarchy for $k\in\{1,2,3\}$. The Helmholtz resonator computational problem lies in $\Delta_2$.}
\label{fig:sci-hir}
\end{figure}
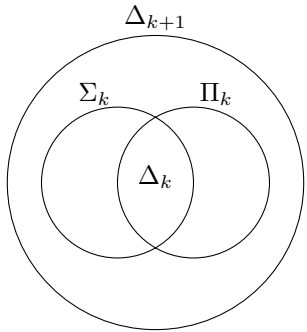

\subsection*{What about our computational problem?}
In view of Definitions \ref{1st_SCI} and \ref{def:pi-sigma}, and considering that  Theorem \ref{th:mainth} amounts to showing that our computational problem lies in $\Delta_2$, one is naturally led to ask whether this can be improved: is the computational problem in $\Sigma_1$? $\Pi_1$? or perhaps even $\Delta_1$?

To show any such results, one would have to keep track of all the constants appearing throughout the proofs and bound them \emph{a priori}. While this might be possible for most constants (given certain \emph{a priori} bounds on the curvature of $\del U$), there is one constant that we cannot estimate: the constant $\tilde C$ appearing in \eqref{eq:analytic_decay}, which controls the width of the zeroes of $\det(I+A)$. 
Due to this, at present there is no hope to obtain error bounds. However, the question of whether (under suitable additional assumptions) one can estimate $\tilde C$ (and all other constants) is the subject of ongoing research.

\bibliography{mybib}
\bibliographystyle{abbrv}

\end{document}